\documentclass[11pt,reqno]{amsart}
\usepackage[utf8]{inputenc}
\usepackage[paperwidth=8.26in,paperheight=11.69in,
left=1.125in, right=1.125in, bottom=1.5in]{geometry}
\usepackage{xcolor}
\newcommand{\edit}[1]{{#1}}

\usepackage{amsmath,amsthm,amssymb}

\usepackage{thmtools}
\declaretheoremstyle[
headfont=\normalfont\bfseries,
headindent= 0pt,
bodyfont=\em,
spaceabove=8pt,
spacebelow=8pt
]{thm}
\declaretheoremstyle[
headfont=\normalfont\em,
headindent= 0pt,
spaceabove=8pt,
spacebelow=8pt
]{remark}
\declaretheoremstyle[
headfont=\normalfont\bfseries,
headindent= 0pt,
spaceabove=8pt,
spacebelow=8pt
]{example}
\declaretheoremstyle[
headfont=\normalfont\bfseries,
headindent= 0pt,
spaceabove=8pt,
spacebelow=8pt
]{definition}

\declaretheorem[name=Theorem,style=thm,numberwithin=section]{theorem}
\newtheorem*{thm*}{Theorem}
\declaretheorem[name=Proposition,style=thm,sibling=theorem]{proposition}
\declaretheorem[name=Lemma,style=thm,sibling=theorem]{lemma}
\declaretheorem[name=Corollary,style=thm,sibling=theorem]{corollary}
\declaretheorem[name=Example,style=definition,sibling=theorem]{example}

\declaretheorem[name=Remark,style=remark]{remark}

\declaretheorem[name=Definition,style=definition,sibling=theorem]{definition}

\usepackage{cleveref}
\crefname{theorem}{Theorem}{Theorems}
\crefname{proposition}{Proposition}{Propositions}
\crefname{lemma}{Lemma}{Lemmas}
\crefname{corollary}{Corollary}{Corollaries}
\crefname{example}{Example}{Examples}
\crefname{definition}{Definition}{Definitions}
\crefname{remark}{Remark}{Remarks}
\crefname{section}{Section}{Sections}
\crefname{enumi}{}{}
\crefname{equation}{}{}

\numberwithin{equation}{section}
\newcommand{\bC}{\mathbb{C}}

\newcommand{\dvol}{d\mathrm{vol}}
\newcommand{\boxop}{\widetilde{\Box}}
\newcommand{\note}[1]{#1}

\DeclareMathOperator{\dom}{dom}
\DeclareMathOperator{\image}{im}

\begin{document}
\title[The $\partial$-complex on weighted Bergman spaces]{The $\partial$-complex on weighted Bergman spaces on Hermitian manifolds}

\author{Friedrich Haslinger}
\address{Fakultät für Mathematik, Universität Wien, Oskar-Morgenstern-Platz 1, 1090 Wien, Austria}
\email{friedrich.haslinger@univie.ac.at}
\author{Duong Ngoc Son}
\address{Fakultät für Mathematik, Universität Wien, Oskar-Morgenstern-Platz 1, 1090 Wien, Austria}
\email{son.duong@univie.ac.at}
\date{January 10, 2020}
\thanks{The first-named author was partially supported by the Austrian Science Fund, FWF-Projekt P 28 154-N35. The second-named author was supported by the Austrian Science Fund, FWF-Projekt M 2472-N35.}
\begin{abstract}
	In this paper, we investigate the $\partial$-complex on weighted Bergman spaces on Hermitian manifolds satisfying a certain holomorphicity/duality condition. This generalizes the situation of the Segal-Bargmann space in $\mathbb{C}^n$, studied earlier by the first-named author, in which the adjoint of the differentiation is the multiplication by $z$. The results are applied to two important examples in the unit ball, namely, the complex hyperbolic metric and a conformally K\"ahler metric which are related to Bergman spaces with  so-called ``exponential'' and ``standard'' weights, respectively. In particular, we obtain new estimates for the solutions of the $\partial$-equation on these weighted Bergman spaces.
	
	\medskip
		
	
\end{abstract}
\maketitle
\section{Introduction}
In a recent paper \cite{haslinger}, the author studies the 
$\partial$-complex on the Segal-Bargmann spaces of $(p,0)$-forms
\begin{equation*}
A^2_{(p,0)}(\mathbb{C}^n, e^{-|z|^2}) : = \left\{u = \sum_{|J| = p}{}' u_J dz^J \colon \int_{\mathbb{C}^n} |u|^2 e^{-|z|^2} \, d \lambda < \infty, u_J \ \text{are holomorphic} \right\}.
\end{equation*}
Consider the operator
$$\partial f = \sum_{j=1}^n \frac{\partial f}{\partial z_j}\, dz_j,$$
which is densely defined on $A^2(\mathbb C^n, e^{-|z|^2})$ and maps to $A^2_{(1,0)}(\mathbb C^n, e^{-|z|^2})$,
the space of $(1,0)$-forms with coefficients in $A^2(\mathbb C^n, e^{-|z|^2}).$
We will choose the domain ${\text{dom}}(\partial)$ in such a way that $\partial$ becomes a closed operator on 
$A^2(\mathbb C^n, e^{-|z|^2}).$
 In general, we get the $\partial$-complex
$$  A^2_{(p-1,0)}(\mathbb{C}^n, e^{-|z|^2}) 
\underset{\underset{\partial^* }
\longleftarrow}{\overset{\partial }
{\longrightarrow}} A^2_{(p,0)}(\mathbb{C}^n, e^{-|z|^2}) \underset{\underset{\partial^* }
\longleftarrow}{\overset{\partial }
{\longrightarrow}} A^2_{(p+1,0)}(\mathbb{C}^n, e^{-|z|^2}),
$$ 
where $1\leqslant p \leqslant n-1$ and $\partial^*$ denotes the adjoint operator of $\partial.$

It is well-known that the forms with polynomial coefficients are dense in the Segal-Bargmann 
space and hence $\partial$ is a densely defined operator on $A^2_{(p,0)}(\mathbb{C}^n, e^{-|z|^2})$. Furthermore, it is proved in \cite{haslinger} that the
associated complex Laplacian 
$$\tilde \Box_p = \partial^* \partial + \partial \partial^*,$$
with ${\text{dom}} (\tilde\Box_p) =\{ f\in {\text{dom}}(\partial ) \cap {\text{dom}}(\partial^* ) : \partial f \in {\text{dom}}(\partial ^*) \ {\text{and}} \ \partial f^* \in {\text{dom}}(\partial )\}$ 
 acts as \edit{an} unbounded self-adjoint operator on $A^2_{(p,0)}(\mathbb{C}^n, e^{-|z|^2}).$ 
It has a bounded, even compact inverse $\widetilde{N}_p$. This exposes an important difference between the $\partial$-complex and the well-known $\overline{\partial}$-complex on the weighted $L^2$-space with the same weight function, where the corresponding complex Laplacian fails to \edit{have a compact} resolvent (see \cite{haslinger/book}).

The inspiration for \cite{haslinger} comes from quantum mechanics, where the annihilation operator $a_j$ can be represented by differentiation with respect to $z_j$ on $A^2(\mathbb C^n, e^{-|z|^2})$ and its adjoint, the creation operator $a^*_j,$ can be represented by multiplication by $z_j,$ both operators being unbounded and densely defined (see \cite{folland/book}). One can show that $A^2(\mathbb C^n, e^{-|z|^2})$ with this action of the $a_j$ and $a^*_j$ is an irreducible representation $M$ of the Heisenberg group, by the Stone-von Neumann theorem it is the only one up to unitary equivalence. Physically $M$ can be thought of as the Hilbert space of a harmonic oscillator with $n$ degrees of freedom and Hamiltonian operator
\begin{equation*}
H= \sum_{j=1}^n \frac{1}{2} (P_j^2+Q_j^2) = \sum_{j=1}^n \frac{1}{2} (a_j^* a_j+a_ja_j^*).
\end{equation*}

The duality between differentiation and multiplication is a special feature of the Segal-Bargmann space: For general weighted spaces of holomorphic functions, we cannot expect such a duality. This poses the first question that is studied in this paper: \textit{Under what conditions on a Hermitian metric $h_{j\bar{k}} dz^j \otimes d\bar{z}^k$ and a weight function $e^{-\psi}$ on a given complex manifold $M$, does the corresponding $\partial$-complex on the weighted Bergman spaces of $(p,0)$-forms possess a similar duality?}  To this end we exhibit conditions on the weight and the Hermitian metric under which the $\partial$-complex has this desired duality.
In the K\"ahlerian case, the result can be stated as follows. 
\begin{theorem}[$=$ \cref{thm:2.10}]\label{thm:1.1a}
	Let $(M,h)$ be a complete Kähler manifold with weight $e^{-\psi}.$  Assume that $\partial_p$ is densely defined in the Bergman space $A^2_{(p,0)}(M,h,e^{-\psi})$. If $(\bar{\partial} \psi)^{\sharp}$ is holomorphic, then 
	\begin{equation}\label{e:2.47}
	D^{\ast} \eta = \partial^{\ast}\eta
	\end{equation} 
	for all $\eta \in \dom(D^{\ast}) \cap A^2_{(p + 1,0)}(M,h,e^{-\psi})$. In particular, $D^{\ast} \eta$ is holomorphic for holomorphic $\eta \in \dom(D^{\ast})$. Here, $D^{\ast}$ and $\partial^{\ast}$ are the Hilbert space adjoints of $\partial$ in the Lebesgue space $L^2_{(p + 1,0)}(M,h,e^{-\psi})$ and $A^2_{(p + 1,0)}(M,h,e^{-\psi})$, respectively.
\end{theorem}
\begin{remark}
\begin{enumerate}
\renewcommand{\labelenumi}{(\roman{enumi})}
	\item In some situations, the completeness of the metric is not needed for the study of the $\partial$-operator on the weighted Bergman spaces; see Section 5.2.
	\item The condition $(\bar{\partial} \psi)^{\sharp}$ being holomorphic means that in any local coordinate patch $(U,z)$, $h^{l\bar{k}} \partial_{\bar{k}} \psi$ (summation convention) is holomorphic for each $l$. This condition also appears in several similar (but not directly related) problems (cf. \cite{gross1999hypercontractivity,munteanu--wang}). For the general non-K\"ahlerian case, the required conditions involve the torsion tensor which is rather complicated, especially for higher degree forms. 
\end{enumerate}
\end{remark}

In view of \cref{thm:1.1a}, we focus on the Hermitian manifolds satisfying these holomorphicity/duality conditions, in both K\"ahlerian and non-K\"ahlerian cases, and prove the coercivity of $\boxop$ and its inverse, the Neumann operators $\widetilde{N}$. These properties are important for our applications.In fact, these results are applied to the $\partial$-equation in the unit ball with a given right-hand side belonging to a suitable Bergman space. We will consider the complex hyperbolic metric on the unit ball $\mathbb{B} := \{z \in \mathbb{C}^n \colon |z| ^2 < 1\}$ and a conformally flat (non-K\"ahler when $n\geqslant 2$) metric with appropriate weights. These two models have close relations to the weighted Bergman spaces with the so-called ``exponential'' and ``standard'' weights. For instance, in the latter case the weighted Bergman space is
\[
	A^{2}_{\gamma} : = \left\{f \in \mathcal{O}(\mathbb{B})
	\colon \|f\|^2_{\gamma}:=\int_{\mathbb{B}} |f|^2 (1-|z|^2)^{\gamma} d\lambda < \infty \right\}, \quad \gamma > -1.
\]
Here, $d\lambda$ is a (constant multiple) of the standard Lebesgue measure on $\mathbb{C}^n$. The coercivity of $\widetilde{\Box}$ for $(1,0)$-forms implies that the canonical solution to the $\partial$-equation exists in weighted Bergman spaces with a sharp estimate. In the case of the second model, we obtain the following result.
\begin{theorem}\label{thm:1.1}
	Suppose that $\gamma > 0$.	For every $\eta_1, \eta_2, \dots, \eta_n \in A^2_{\gamma}(\mathbb{B})$ such that $\partial \eta_j/\partial z^k = \partial \eta_k/\partial z^j$ for every pair $j,k = 1,2, \dots, n$, there exists $f \in A^2_{\gamma}(\mathbb{B})$ such that $\partial f/\partial z^k = \eta_k$ for every $k= 1,2, \dots, n$, and 
	\begin{equation} \label{e:1.1}
	\int_{\mathbb B} |f|^2 (1- |z|^2)^{\gamma-1} d\lambda \leqslant \frac{1}{\gamma} \, \int_{\mathbb{B}} \sum_{k=1}^n |\eta_k|^2 (1-|z|^2)^{\gamma}\, d\lambda.
	\end{equation}
	The constant $\gamma^{-1}$ on the right-hand side of \cref{e:1.1} is sharp. The equality occurs if and only if $\eta_k$'s are constants.
\end{theorem}

\begin{remark}
 Our approach is based on an explicit analysis of the complex Laplacian $\widetilde{\Box}$ (which in our situations reduces to a first order differential operator) using suitable orthonormal bases of the weighted Bergman spaces of $(1,0)$-forms. This allows us to study the spectrum of $\widetilde{\Box}$ via the spectra of a family of (finite) matrices; see the discussion in Section~5 for more details.
\end{remark}
\cref{thm:1.1} follows from a more detailed statement in \cref{thm:5.3}. Using the same strategy, we prove in \cref{prop:51} an analogous result for the Bergman space with exponential weight. We believe that this strategy can be applied to get useful results in other situations as well.

The structure of the paper is as follows. In Section 2, we establish basic properties of the $\partial$-complex on the weighted Bergman spaces on Hermitian manifolds. In particular, in Theorem~2.10  we give a condition on the Hermitian metric $h_{j\bar{k}} dz^j \otimes d\bar{z}^k$ and the weight $\psi$ such that the adjoints of the $\partial$-operator in the $L^2$-space and in the Bergman space agree in the relevant domains. In Section 3, we briefly discuss the complex Laplacian and a basic identity generalizing the case of Segal-Bargmann space in \cite{haslinger} and some well-known basic identity in the $L^2$-theory for the $\overline{\partial}$-complex. In Section 4, we develope the theory of the $\partial$-Neumann operator adapted to the case of $\partial$-complex on weighted Bergman spaces. Finally, in Section 5, we investigate the $\partial$-complex on weighted Bergman spaces on the unit ball and solve the $\partial$-equation, proving \cref{thm:1.1} and an analogous result for the Bergman space with exponential weight.
\section{The $\partial$-operators on weighted Bergman spaces}
In this section, we study some general properties of the $\partial$-complex on the Bergman spaces on Hermitian manifolds. For the reader's convenience, we recall here some basic facts and fix some notations. For general references regarding Hermitian manifolds and the $\partial$-complex, we refer to \cite{kodaira--morrow,andreotti1965carleman} and \cite{haslinger/book}, respectively.

Let $(M,h)$ be a Hermitian manifold. In holomorphic coordinates $z^1, \dots , z^n$, the metric $h$ has the form
\begin{equation}
h_{j\bar{k}} dz^j \otimes dz^{\bar{k}},
\end{equation}
where $\left[h_{j\bar{k}}\right]$ is a positive definite Hermitian matrix with smooth coefficients. This metric
induces a volume element which we denoted by $\dvol_h$. If $\psi$ is a weight function on $M$, then the Hilbert space of $L^2$ integrable functions with respect to the 
measure $d\mu : = e^{-\psi} \dvol_h$ is defined by
\begin{equation}
L^2(M, e^{-\psi} \dvol_h)
=
\left\{
f\colon M \to \bC\ \text{measurable}\ \colon \int_M |f|^2 e^{-\psi} \dvol_h < +\infty
\right\}.
\end{equation}
The weighted Bergman space {with weight $\psi$} is defined to be
\begin{equation}
A^2(M, e^{-\psi} \dvol_h) = L^2\left(M, e^{-\psi} \dvol_h \right) \cap \mathcal{O}\, (M). 
\end{equation}
Here, $\mathcal{O}\, (M)$ denotes the space of holomorphic functions on $M$.
Under a suitable condition on $\psi$, the Bergman space $A^2(M,e^{-\psi} \dvol_h)$ is a closed subspace of $L^2(M,e^{-\psi} \dvol_h)$ and thus it is a Hilbert space (although it can be trivial, finite, or infinite dimensional.)

The Hermitian metric $h$ induces a metric on tensors of every degree. For example, 
if in {local coordinates} $u= u_j dz^j$ and $v = v_j dz^j$ are $(1,0)$-forms, then 
\begin{equation}
\langle u , v \rangle_h = h^{j \bar{k}} u_j v_{\bar{k}}, \quad |u|^2_h =\langle u , u \rangle_h \quad
\end{equation}
where $\left[ h^{j \bar{k}}\right] $ is the transpose of the inverse matrix of $\left[h_{j \bar{k}}\right]$. We define the weighted spaces of $(p,0)$-forms
\begin{equation}
L^2_{(p,0)}(M, h, e^{-\psi})
=
\left\{
u \text{ is a } (p,0)\text{-form} \colon \int_M |u|_h^2\, e^{-\psi} \dvol_h < \infty
\right\}, \quad 0 \leqslant p \leqslant n,
\end{equation}
with inner product
\begin{equation}
(u,v)_{h,\psi} = \int_M \langle u ,v \rangle_h \, e^{-\psi} \dvol_h .
\end{equation}

We say that a $(p,0)$-form $u$ is holomorphic if in local holomorphic coordinates, we can write
\begin{equation}
u = \sum_{|J|=p}{}^{'} u_J dz^{J}
\end{equation} 
with holomorphic coefficients $u_J$ and with summation over increasing multiindices. Observe that this notion does not depend on the
chosen coordinates (cf. \cite{haslinger}) and hence is well-defined on complex manifolds. We define the Bergman space of $(p,0)$-forms to be
\begin{equation}
A^2_{(p,0)}(M, h, e^{-\psi})
=
\left\{
u \text{ is a holomorphic } (p,0)\text{-form } \colon \int_M |u|_h^2\, e^{-\psi} \dvol_h < \infty
\right\}.
\end{equation}
For smooth forms, the $\partial$-operator is defined in local coordinates by
\begin{equation}
\partial u
:=
\sum_{|J| = p}{}^{\prime} \sum_{j=1}^n \frac{\partial u_J}{\partial z_j} dz^j \wedge dz^{J}.
\end{equation}
Thus, if $u$ is holomorphic, then so is $\partial u$. 

For a $(p,0)$-form $u$ in $A^2_{(p,0)}(M, h, e^{-\psi})$, 
it is not necessary that the $(p+1,0)$-form $\partial u$ is in $A^2_{(p+1,0)}(M, h, e^{-\psi})$. Therefore, we introduce the subspace
\begin{equation}
\dom(\partial_p)
=
\left\{
u\in A^2_{(p,0)}(M, h, e^{-\psi}) \colon \partial u \in A^2_{(p+1,0)}(M, h, e^{-\psi})
\right\}.
\end{equation}
Clearly, $\dom(\partial_p)$ also depends on both the metric $h$ and the weight function $\psi$.

The interesting situation is when $\dom(\partial_p)$ is dense in $A^2_{(p,0)}(M, h, e^{-\psi})$, for each $p$.
In this case, $\partial$ is a densely defined (bounded or unbounded) operator:
\begin{equation*}
\partial_p\colon A^2_{(p,0)}(M, h, e^{-\psi}) \to A^2_{(p+1,0)}(M, h, e^{-\psi}),
\end{equation*}
and the powerful theory of unbounded operators applies.

Although for general Hermitian manifolds, it is difficult to determine when $\dom(\partial_p)$ restricted to the weighted Bergman space is dense, this is the case in many interesting situations.
\begin{example}
	Let $M = \mathbb{C}^n$ and suppose that $h$ is the standard Euclidean metric and $\psi \colon \mathbb{C}^n \to \mathbb{R}$ is convex as a function of $2n$ real variables (e.g., $\psi(z) = |z|^2$ satisfies this convexity assumption).
	By a result of B.A. Taylor \cite{taylor}, the polynomials
	are dense in $A^2(\mathbb{C}^n, e^{-\psi}d\lambda)$, provided that $A^2(\mathbb{C}^n, e^{-\psi}d\lambda)$ contains the polynomials. More generally, $(p,0)$-forms with polynomial coefficients are dense in $A^2_{(p,0)}(\mathbb{C}^n, e^{-\psi}d\lambda)$. In this case, since the $\partial$-operator sends $(p, 0)$-forms with polynomial coefficients to $(p+1,0)$-forms with polynomials coefficients, $\partial$ is densely defined on $A^2_{(p,0)}(\mathbb{C}^n, e^{-\psi}d\lambda)$. The case $\psi(z) = |z|^2$ {corresponds to the Segal--Bargmann space and has been treated thoroughly} in \cite{haslinger}. Similarly, if for some $\psi$ all the exponentials are dense in $A^2(\mathbb{C}^n, e^{-\psi}d\lambda)$, then $\dom(\partial)$ is also dense in $A^2(\mathbb{C}^n, e^{-\psi}d\lambda)$.
\end{example}

In the next two propositions, we establish the relation between the $\partial$-operators on the weighted Bergman spaces and on the weighted $L^2$ spaces. We denote by $D_p$ the maximal extension (in the sense of distributions) of the $\partial$-operator acting on $L^2_{(p,0)}(M, h, e^{-\psi}\dvol_h)$.
\begin{proposition} Let $(M,h,\psi)$ be as above. Then for each $p\geqslant 0$, it holds that
	\begin{equation}\label{e211}
	\dom(\partial_p) 
	=
	\dom(D_p) \cap A^2_{(p,0)}(M, h, e^{-\psi}). 
	\end{equation}
\end{proposition}
\begin{proof}
	If $u\in \dom(\partial) \subset A^2_{(p,0)}(M,h, e^{-\psi})$, then $u$ has holomorphic coefficients and $\partial u \in A^2_{(p+1,0)}(M,h, e^{-\psi}) \subset L^2_{(p+1,0)}(M, h, e^{-\psi})$. Hence $u\in \dom(D)\cap A^2_{(p,0)}(M, h, e^{-\psi})$, {as desired}. Conversely, if $u$ belongs to the right hand side of \cref{e211}, then $u$ has holomorphic coefficients and \note{$|\partial u|_{h}$} is $L^2$-integrable with respect to $d\mu$. This clearly implies $u\in \dom(\partial)$. \note{The proof is complete.}
\end{proof}

Suppose that $\dom(\partial)$ is dense in $A^2_{(p,0)}(M, h,e^{-\psi})$. Then the Hilbert space adjoint of $\partial$ in $A^2_{(p,0)}(M, h, e^{-\psi})$ is well-defined and denoted by $\partial^{\ast}$ (see \cite[Definition 4.1]{haslinger/book}). Clearly,
\begin{equation*}
\dom(\partial^{\ast})
=
\left\{
g \in A^2_{(p+1,0)}(M, h, e^{-\psi}) \colon 
f \mapsto( \partial f, g)_{h,\psi} \ \text{is continuous on } \dom(\partial)
\right\}.
\end{equation*}
Since $(M,h)$ is complete, $D$ is densely defined in $L^2_{(p,0)}(M, h,e^{-\psi})$ and hence the Hilbert space adjoint $D^{\ast} $ of $D$ is well-defined. Assume that $g \in \dom(D^{\ast})$, then $f \mapsto ( \partial f, g )_{h,\psi}$
is continuous on $\dom(D)$ and hence on $\dom(\partial)$ since $\dom(\partial) \subset \dom(D)$. Thus, we obtain
\begin{proposition} Suppose that $\dom(\partial)$ is dense in $A^2_{(p,0)}(M, h,e^{-\psi})$. Then
	\begin{equation}\label{e:domdom}
	\dom(D^{\ast}) \cap A^2_{(p+1,0)}(M,h, e^{-\psi})
	\subset 
	\dom(\partial^{\ast}).
	\end{equation}
\end{proposition}
It is natural to ask when is the left hand side of \cref{e:domdom} dense in the right hand side? Suppose that $v$ is a $(p+1,0)$-form in $\dom(\partial^{\ast})$; in particular, $v$ has holomorphic coefficients. Then $v\in \dom(D^{\ast})$ if and only if
$u \mapsto ( \partial u, v )_{h,\psi} = ( P_{h,\mu}(\partial u) , v )_{h,\psi}$ is continuous on $\dom(D)$, where $P_{h,\psi}$ is the Bergman orthogonal projection
\begin{equation}\label{e:bp}
P_{h,\psi} \colon L^2_{(p,0)}(M,h, e^{-\psi}) \longrightarrow A^2_{(p,0)}(M,h, e^{-\psi}),
\end{equation}
which is well-defined under the admissibility condition in the sense of \cite{pasternak-winiarski} (the map is \textit{a priori} continuous on the subspace $\dom(D) \cap A^2(M, h, e^{-\psi}))$.
\begin{proposition}
	Suppose that $\dom(\partial)$ is dense in $A^2_{(p,0)}(M,h, e^{-\psi})$. Then the operators $\partial$ and $\partial^{\ast}$ are closed operators.
\end{proposition}
\begin{proof} By the general theory for closed unbounded operators (see, e.g., \cite[Lemma~4.5]{haslinger/book}), we only need to prove the statement for $\partial$. Suppose that $\{u_j\}_j$ is a sequence in $A^2_{(p,0)}(M,h, e^{-\psi})$ that converges 
	to $u$ in $L^2$-topology and suppose that $\partial u_j$ converges to $v$ \note{also in $L^2$-topology}. \note{ Since $d\mu$ is a positive measure with smooth positive density with respect to the Lebesgue measure in any local coordinate patch $U$, it is ``admissible'' in the sense of \cite{pasternak-winiarski}}. Thus, the coefficients of $u_j$ converge
	to those of $u$ uniformly on compact sets of $U$ and hence $\partial u = v$. Since $v\in A^2_{(p+1,0)}(M,h, e^{-\psi})$ by assumption, we obtain that $u\in \dom(\partial)$. The proof is complete.
\end{proof}
Thus, if $\dom(\partial)$ is dense in $A^2_{(p,0)}(M,h, e^{-\psi})$, then $\dom(\partial^{\ast})$ is dense in $A^2_{(p+1,0)}(M,h, e^{-\psi}),$ see, e.g., \cite[Lemma~4.6]{haslinger/book}.
\begin{proposition}\label{prop:25}
	Suppose that $\dom(\partial)$ is dense in $A^2_{(p,0)}(M, h, e^{-\psi})$ and $v\in \dom(\partial^{\ast})$.
	If $w \in L^2_{(p,0)}(M, h, e^{-\psi})$ such that
	\begin{equation}\label{e:1}
	\left( \partial u, v \right)_{h,\psi} = \left( u, w \right)_{h,\psi} ,\quad \forall\, u \in \dom(\partial),
	\end{equation}
	then
	\begin{equation}\label{e:2a}
	\partial^{\ast} v = P_{h,\psi}(w).
	\end{equation}
	In particular, if $v \in \dom(D^{\ast}) \cap A^2_{(p+1,0)}(M, h, e^{-\psi})$ then
	\begin{equation}\label{e:3}
	\partial^{\ast} v = P_{h,\psi}\left( D^{\ast} v \right).
	\end{equation}
\end{proposition}
\begin{proof} 
	Since $v\in \dom(\partial^{\ast})$, then for all $u\in \dom(\partial)$, one has $\left( u, \partial^{\ast} v \right)_{h,\psi} = \left( \partial u, v \right)_{h,\psi}$. Thus
	\begin{equation}
	\left( u, w \right)_{h,\psi} = \left( u, \partial^{\ast} v \right)_{h,\psi}, \quad \forall\ u\in \dom(\partial).
	\end{equation}
	Thus, by the density of $\dom(\partial)$ in $A^2_{(p,0)}(M, h, e^{-\psi})$ and the definition of
	the Bergman projection, \cref{e:2a} follows. If, in addition, $v \in \dom(D^{\ast})$, then \cref{e:1} holds with
	$w = D^{\ast} v$ and hence \cref{e:3} follows.
\end{proof}
We point out that \cref{e:3} is an extension of equation (3.3) in \cite{haslinger} in which \cref{e:1} was verified using the Green-Gauß theorem.
We shall improve this result in \cref{prop:31}.

If $(M, h)$ is a Hermitian manifold, then there exists a canonical linear connection on $M$, the Chern connection of $h$, which parallelizes both the metric $h$ and the complex structure of the underlying manifold (see, e.g., \cite{kodaira--morrow,griffiths}). In local coordinates $z^1,\dots , z^n$, the nonvanishing Christoffel symbols for the Chern connection are 
\begin{equation}
\Gamma^i_{jk} = h^{i\bar{l}} \partial_j h_{k\bar{l}}, \quad \Gamma^{\bar{i}}_{\bar{j}\bar{k}} = \overline{\Gamma^i_{jk}}.
\end{equation}
Since $\partial_j \left(h^{i\bar{l}} h_{k\bar{l}}\right) = \partial_j(\delta_{ik}) = 0$, we also have $\Gamma^{i}_{jk} = - h_{k\bar{l}}\partial_j h^{i\bar{l}}
$. The covariant derivatives can be explicitly expressed in local coordinates. For examples, if in local coordinate $u = u_k dz^k$ is a $(1,0)$-form, then
\begin{equation}\label{e:cd}
\nabla_j u_k = \partial_j u_k - \Gamma^l_{jk} u_l,
\quad
\nabla_{\bar{j}} u_k = \partial_{\bar{j}} u_k.
\end{equation}
Note that in the second equation, the Christoffel symbols of ``mixed type'' vanish and hence the covariant derivative of $(1,0)$-forms along $(0,1)$-direction reduces essentially to the partial derivatives of its components.

For a general Hermitian metric, the torsion tensor may be nontrivial; we define the torsion $T^i_{jk}$ by
\begin{equation}
T^i_{jk} = \Gamma^i_{jk} - \Gamma^i_{kj}, \quad T^{\bar{i}}_{\bar{j}\bar{k}} = \overline{T^i_{jk}}
\end{equation}
The torsion $(1,0)$-form is then obtained by taking the trace:
\begin{equation}
\tau = T^{i}_{ji} d z^j.
\end{equation}
In local coordinates, the volume element is given by $\dvol_h = \det (h_{j\bar{l}})\, d\lambda$, where $d\lambda$ is the Lebesgue measure in that coordinate patch.
Thus, if $\psi$ is an weight function, then we can write (locally) $d\mu = e^{-\psi}\dvol = e^{-\varphi} d\lambda$, with
\begin{equation}
\varphi = \psi - \log \det (h_{j\bar{l}}).
\end{equation}
Therefore, since $\Gamma^{\bar{i}}_{\bar{k}\bar{i}} = \partial _{\bar{k}} \log \det (h_{j\bar{l}})$ (see, e.g., \cite[p. 111]{kodaira--morrow}, {but mind that the K\"{a}hlerian condition is not assumed here}), we obtain
\begin{equation}\label{e:225}
\varphi_{\bar{k}}
=
\psi_{\bar{k}} - \partial _{\bar{k}} \log \det (h_{j\bar{l}}) = \psi_{\bar{k}} - \Gamma^{\bar{i}}_{\bar{k} \bar{i}} .
\end{equation}
Suppose that $u = u_j dz^j$ is a smooth $(1,0)$-form and $v$ is compactly supported function. We assume that $v$ has support contained in a coordinate patch $(U, z)$. Then 
\begin{align}
( u, \partial v )_{h,\psi}
& =
\int_U u_j \overline{v_k}\, h^{j\bar{k}} e^{-\varphi} d\lambda \notag \label{e37}\\
& =
-\int_U \partial_{\bar{k}} \left(u_j h^{j\bar{k}} e^{-\varphi}\right) \overline{v}\, d\lambda \notag \\
& = - \int_M \left(\partial_{\bar{k}} u_j h^{j\bar{k}} + u_j \partial_{\bar{k}} h^{j\bar{k}} - u_j\varphi_{\bar{k}}h^{j\bar{k}}\right) \overline{v}\, d\mu. 
\end{align}
where $v_k := \frac{\partial v}{\partial z^k}$. Observe that
\begin{align}\label{e:227}
\partial_{\bar{k}} h^{j\bar{k}} 
= - h^{j\bar{k}} \Gamma^{\bar{l}}_{\bar{l}\bar{k}} 
& = h^{j\bar{k}} \left( \tau_{\bar{k}} - \Gamma^{\bar{l}}_{\bar{k}\bar{l}} \right) \notag \\
& = h^{j\bar{k}} \left( \tau_{\bar{k}}- \partial_{\bar{k}} \log \det(h_{l\bar{k}})\right) \notag \\
& = h^{j\bar{k}} \left( \tau_{\bar{k}} + \varphi_{\bar{k}} - \psi_{\bar{k}} \right).
\end{align}
Plugging this into \cref{e37}, we obtain the integration-by-part formula:
\begin{equation}\label{e:ibp}
\left( u, \partial v\right)_{h,\psi}
=
- \int_M h^{j\bar{k}} \left[\partial_{\bar{k}} u_j + u_j (\tau_{\bar{k}} - \psi_{\bar{k}})\right] \bar{v}\, d\mu. 
\end{equation}
For the case of general compactly supported $v$, we can use the partition of unity to reduce to the case above; we omit the details.

Thus, if additionally $(M,h)$ is a complete manifold (so that the Andreotti--Vesentini density lemma applies) and $u\in \dom(D^{\ast})$, then we have a local expression for $D^{\ast} u$ as follows (see, e.g.,  \cite[(6.20)]{griffiths} or \cite{berger--dallara--son}):
\begin{equation}\label{e39}
D^{\ast} u 
=
- \nabla_{\bar{k}} u^{\bar{k}} + (\psi_{\bar{k}} - \tau_{\bar{k}}) u^{\bar{k}}, \quad u^{\bar{k}} : = h^{j\bar{k}} u_j.
\end{equation}

We shall derive a similar formula for the adjoint $\partial^{\ast}$ on
$\dom(\partial^{\ast})$. It turns out that, similar to the Segal--Bargmann case \cite{haslinger}, the adjoint $\partial^{\ast}$ is closely related to the Bergman projection \cref{e:bp}, which is nonlocal.
\begin{proposition}\label{prop:31}
	Let $(M,h)$ be a complete Hermitian manifold and $e^{-\psi}$ a smooth weight on~$M$. Suppose that $\dom(\partial)$ is dense in $A^2_{(1,0)}(M,h,e^{-\psi})$ and let $u = u_j dz^j \in A^2_{(1,0)}(M,h, e^{-\psi}).$
	Let $\partial^{\ast}$ be the Hilbert space adjoint of $\partial$.
	If $\langle u , \partial \psi - \tau \rangle_h \in L^2(M,h, e^{-\psi})$, then $u$ belongs to $\dom(D^{\ast})$ and hence $u\in \dom(\partial^{\ast})$. Moreover,
	\begin{equation}\label{e:dbar-formula}
	\partial^{\ast} u = P_{h, \psi}\left(\langle u , \partial \psi - \tau \rangle_h\right).
	\end{equation}
	where $\tau = \tau_{j} dz^{j}$ is the torsion $(1,0)$-form.
\end{proposition}
\begin{remark}
	This proposition implies that if $|\partial \psi - \tau|_h$ is bounded, then $\partial^{\ast}$ is a bounded operator. \cref{ex:44} exhibits a situation with finite dimensional generalized Bergman spaces so that $\partial$ and $\partial^{\ast}$ are bounded operators and $|\partial \psi - \tau|_h$ is bounded. For the Segal-Bargmann model (cf. \cite{haslinger}), $|\partial \psi - \tau|_h = \left| z_j\, dz^j \right|$ is unbounded and so are $\partial$ and $\partial^{\ast}$. See also Section~5 where we exhibit two unbounded examples.
	
	Compared to the local expression \cref{e39}, the formula for $\partial^{\ast}$ in \cref{e:dbar-formula} is global as it involves the Bergman projection.
\end{remark}
\begin{proof}[Proof of \cref{prop:31}]
	We shall use the usual cut-off procedure on complete Riemannian
	manifolds (see, e.g., \cite[pp. 48]{ohsawa/book}).
	For a fixed point $p_0 \in M$, the distance function $d(\cdot, p_0)$ is Lipschitz on $M$. Let $\rho(x)$ be a smoothing 
	of $d(x,x_0)$ and choose a function $\chi \colon \mathbb{R} \to \mathbb{R}$ such that $\chi\bigl |_{(-\infty, 1)} \equiv 1$ and $\mathrm{supp}\, \chi \subset (-\infty, 2]$. Put
	\begin{equation}
	\chi_R (x) = \chi\left(\rho(x)/R\right), 
	\quad R > 1.
	\end{equation}
	Then $\chi_R$ has compact support and $|\partial \chi_R| \leqslant c/R$ for some $c>0$.
	Suppose that $u = u_j dz^j \in A^2_{(1,0)}(M,h, e^{-\psi})$ and $v\in \dom(D)$.
	In local coordinates, $\partial v = v_k dz^k$ where $v_k = \partial v/\partial z^k$. Using integration by parts \cref{e:ibp}
	\begin{align}
	\left( \chi_R\, u , \partial v\right)_{h,\psi}
	& =
	-\int_M h^{j\bar{k}} \left[\partial_{\bar{k}} (\chi_R\, u_j) + \chi_R u_j (\tau_{\bar{k}} - \psi_{\bar{k}}) \right] \bar{v}\, d\mu \notag \\
	& =
	- \int_M h^{j\bar{k}} u_j (\partial_{\bar{k}} \chi_R)\, \bar{v} \, d\mu + \int_M h^{j\bar{k}} \chi_R u_j (\tau_{\bar{k}} - \psi_{\bar{k}}) \, \bar{v}\, d\mu
	\end{align}
	Here we use $\partial_{\bar{k}} u_j = 0$ since $u$ is holomorphic.
	
	Observe that the first integral tends to $0$ as $R\to \infty$. Indeed,
	\begin{equation}
	\left|\int_M h^{j\bar{k}} u_j (\partial_{\bar{k}} \chi_R)\, \bar{v} \, d\mu\right|
	\leqslant
	\int_M |v|\, |u|_h\, |\partial \chi_R|_h \, d\mu \to 0 \quad \text{as}\ R\to \infty
	\end{equation}
	since $|\partial \chi_R| < c/R$. Letting $R\to \infty$, we obtain that 
	\begin{equation}
	( u , \partial v )_{h,\psi}
	=
	\lim_{R\to \infty}\int_M \chi_R \bar{v} \langle u, {\partial} \psi - \tau \rangle_h\, d\mu.
	\end{equation}
	Thus, if $\langle u , \partial \psi - \tau \rangle_h \in L^2(M,h, e^{-\psi})$, then the limit on the right hand side is 
	\begin{align}
	\lim_{R\to \infty}\int_M \chi_R \bar{v} \langle u,\partial \psi -\tau \rangle_h\, d\mu
	& =
	\int_M \bar{v} \langle u , \partial \psi - \tau \rangle_h \, d\mu \notag \\
	& =
	\int_M \bar{v} P_{h,\psi}(\langle u , \partial \psi - \tau \rangle_h) d\mu,
	\end{align}
	since $v$ is holomorphic. Therefore, $u\in \dom(\partial^{\ast})$ and
	\begin{equation}
	\partial^{\ast} u = P_{h,\psi}\left(\langle u , \partial \psi - \tau \rangle_h\right).
	\end{equation}
	The proof is complete.
\end{proof}

In view of this result, we call $\delta u := P_{h,\psi}\left(\langle u , \partial \psi - \tau \rangle_h\right)$ the \textit{formal} adjoint of $\partial$, whenever the right hand side is defined.

\begin{example}[cf. \cite{haslinger}]
	Let $M = \mathbb{C}^n$ with the standard Euclidean metric and $\psi (z) = |z|^2$. Let $u_jdz^j\in A^2_{(1,0)}(\mathbb{C}^n, e^{-|z|^2})$. If $\sum_{j=1}^{n} z^j u_j \in L^2(\mathbb{C}^n, e^{-|z|^2})$ then $u \in \dom(\partial^\ast)$ 
	and $\partial^\ast u = \sum_{j=1}^{n} z^j u_j$.
\end{example}
Next, we give a condition under which $\partial^{\ast}$ agrees with $D^{\ast}$ for $u\in \dom(D^{\ast})$ having holomorphic coefficients. To this end, the following notion is crucial for us:
\begin{definition}
	Suppose that $\xi = \xi^{j} \partial_{j}$ is a $(1,0)$-vector field expressed in a local coordinate patch $U$. We say that $\xi$ is \textit{holomorphic} if each coefficient $\xi^j$ is holomorphic in $U$.
\end{definition}
The notion of a holomorphic $(1,0)$-vector field does not depend on the choice of coordinate. For a $(0,1)$ form $w = w_{\bar{k}} \, d\bar{z}^k$, the ``musical operator'' $\sharp$ acts on $w$ and produces an $(1,0)$ vector field $w^{\sharp} : = h^{k\bar{j}} w_{\bar{j}} \, \partial_{k}$.
If $u$ and $v$ are $(1,0)$ forms, then $\langle u, v\rangle_h = (u, \overline{v}^{\sharp})$ where the right hand side is the dual pairing between vectors and covectors. Thus, if $u \in A^2_{(1,0)}(M, h, e^{-\psi})$ and $\overline{v}^{\sharp}$ is a holomorphic vector field, then $\langle u, v\rangle _h$ is a holomorphic function. Thus, we obtain the following
\begin{corollary}\label{cor:33}
	Suppose that $\partial$ is densely defined in the weighted Bergman space and $(\bar{\partial}\psi - \bar{\tau} )^{\sharp}$ is a holomorphic vector field, then for each $u \in \dom(D^{\ast}) \cap A^2_{(1,0)}(M, h, e^{-\psi})$, one has
	\begin{equation}\label{e:2}
	\partial^{\ast} u = D^{\ast} u.
	\end{equation}
\end{corollary}
\begin{proof} If $u \in \dom(D^{\ast}) \cap A^2_{(1,0)}(M, h, e^{-\psi})$, then clearly
	\begin{equation}
	D^\ast u = \langle u , \partial \psi - \tau \rangle_h \in L^2(M,h,e^{-\psi})
	\end{equation}
	and thus $u\in \dom(\partial^\ast)$. 
	
	On the other hand, $(\bar{\partial}\psi - \bar{\tau})^{\sharp}$ is the $(1,0)$-vector field expressed in local coordinates by
	\begin{equation}
	(\bar{\partial}\psi -\bar{\tau} )^{\sharp} = h^{j\bar{k}}\left(\psi_{\bar{k}} - \tau_{\bar{k}}\right)\frac{ \partial}{\partial z^j}.
	\end{equation}
	The holomorphicity to $(\bar{\partial}\psi - \bar{\tau})^{\sharp}$ means that for each $j$, $h^{j\bar{k}}\left(\psi_{\bar{k}} - \tau_{\bar{k}}\right)$ is holomorphic. Thus, by \cref{e:dbar-formula}, $\partial^\ast u = P_{h,\psi}(D^{\ast} u) = D^{\ast} u$. The proof is complete.
\end{proof}
\begin{remark}
	It is worth pointing out that the condition on the holomorphicity of $(\bar{\partial}\psi - \bar{\tau})^{\sharp}$ as in \cref{cor:33} also plays an important role in several (similar, but not directly related) problems in the literature, especially for the K\"ahler case. For example, the condition is necessary and sufficient for the weighted Hodge Laplacian with weight $d\mu: = e^{-\psi} \dvol$ on Kähler manifolds to preserve the type of forms (i.e., to send $(p,q)$-forms into $(p,q)$-forms, see, e.g., \cite{munteanu--wang}). It is also necessary and sufficient for the weighted Dirichlet forms $d^{\ast}_\psi d f$ corresponding to the same weight to be of the form $Zf$ for some holomorphic $(1,0)$-vector field $Z$ (see, e.g., \cite{gross1999hypercontractivity} for more details).
\end{remark}

In the rest of this section, we describe the formula for $\partial^{\ast}$ on $(p,0)$-forms with $p\geqslant 2$. To illustrate the calculations, we start with $p = 2$. For $v \in A^2_{(2,0)}(\mathbb{B}, h, d\mu)$ we write 
\begin{equation}
v = \frac{1}{2}\sum_{j,k} v_{jk} dz^j \wedge dz^k = \sum_{j<k} v_{jk} dz^j \wedge dz^k,
\end{equation}
where $v_{jk} = - v_{kj}$ are holomorphic. If $u = u_{j} dz^j$, we have
\begin{equation}
\partial u
=
\frac{1}{2}\sum_{j,k} \left(\frac{\partial u_k}{\partial z^j} - \frac{\partial u_{j}}{\partial z^k}\right) dz^j \wedge dz^k.
\end{equation}
Moreover, since $v_{pq} = - v_{qp}$, we find that
\begin{align}
\left\langle \partial u, v\right\rangle_h
=
\sum_{j,k,p,q} \overline{v_{pq}}h^{k\bar{p}} h^{j\bar{q}} \left(\frac{\partial u_k}{\partial z^j}\right).
\end{align}
Assuming $(M,h)$ is complete, we use the usual cut-off function technique as in the proof of \cref{prop:31} and apply the integration by parts without boundary terms. The calculations can be done in a local coordinate patch as follows: 
\begin{align}
\left(\partial u, v \right)_{h,\psi}
& =
\int_M \sum_{j,k,p,q} \overline{v_{pq}} \left(\frac{\partial u_k}{\partial z^j}\right) \left(h^{k\bar{p}} h^{j\bar{q}}\right) e^{-\varphi} d\lambda \notag \\
& = - \int_M \sum_{j,k,p,q} \overline{v_{pq}}\biggl(u_k\, \partial_j \left(e^{-\varphi} h^{k\bar{p}} h^{j\bar{q}}\right)\biggr) d\lambda.
\end{align}
Expanding the right hand side in local coordinates using \cref{e:225,e:227}, we obtain
\begin{equation}\label{e:243}
\partial^{\ast} v
=
P_{h,\psi}\left(-(\psi_{\bar{j}} - \tau_{\bar{j}}) v_{pq} h^{q\bar{j}} dz^p + \frac{1}{2} T^{\bar{i}}_{\bar{j}\bar{k}} h^{r\bar{j}} h^{s\bar{k}} v_{rs} h_{p\bar{i}} dz^p \right) .
\end{equation}
Here, $P_{h,\psi}$ is the orthogonal projection from $L^2_{(2,0)}(M, h, e^{-\psi})$ onto $A^2_{(2,0)}(M, h, e^{-\psi})$.

Observe that the participation of the torsion is rather involved in the case $p = 2$. For general $p \geqslant 2$, we follow \cite{griffiths}, which is somewhat explicit. Precisely, if $\eta \in A^2_{(p,0)}(M, h, e^{-\psi})$ is written as $\eta = \frac{1}{p!}\sum_{|I| = p} \eta_I dz^I$, we define, for $p \geqslant 1$, 
\[
T \colon A^2_{(p,0)}(M, h, e^{-\psi}) \to A^2_{(p+1,0)}(M, h, e^{-\psi})
\]
by \cite[(6.9)]{griffiths}
\begin{equation}\label{e251}
T(\eta)
=
\frac{1}{(p-1)!} \sum T^{i}_{jk} \eta_{i i_2 \cdots i_p} dz^j \wedge dz^k \wedge z^{i_2} \wedge \cdots \wedge dz^{i_p}
\end{equation}
and, for $p\geqslant 2$, the ``adjoint'' $T^{\sharp} \colon A^2_{(p,0)}(M, h, e^{-\psi}) \to A^2_{(p-1,0)}(M, h, e^{-\psi})$ by
\begin{equation}
\langle T^{\sharp} \eta , \xi \rangle_h
=
\langle \eta , T \xi\rangle_h.
\end{equation}
If $\eta \in \dom(D^{\ast}) \cap A^2_{(p,0)}(M, h, e^{-\psi})$, then by \cite{griffiths}, and the fact that $\eta$ is holomorphic,
\begin{equation}\label{e:tstar}
D^{\ast} \eta
=
\frac{(-1)^{p-1}}{(p-1)!}h^{j\bar{k}}\psi_{\bar{k}} \eta_{i_1\cdots i_{p-1} j} dz^{i_1} \wedge \cdots \wedge dz^{i_{p-1}} - T^{\sharp}(\eta).
\end{equation}
More generally, if $\eta \in A^2_{(p,0)}$ such that the right hand side of \cref{e:tstar} belongs to $L^2$, then $\eta$ belongs to $ \dom(D^{\ast})$ and hence in $\dom(\partial^{\ast})$. Moreover,
\begin{equation}\label{e254a}
\partial^{\ast} \eta
=
P_{h,\psi}\left(\frac{(-1)^{p-1}}{(p-1)!} h^{j\bar{k}}\psi_{\bar{k}} \eta_{i_1\cdots i_{p-1} j} dz^{i_1} \wedge \cdots \wedge dz^{i_{p-1}} - T^{\sharp}(\eta)\right).
\end{equation}
This generalizes the formula in \cref{prop:31} (cf. \cite[Eq. (3.3)]{haslinger}). The proof uses a calculation as in \cite{griffiths} and a density lemma by Andreotti--Vesentini \cite[Lemma 4]{andreotti1965carleman}. We omit the details.
\begin{theorem}\label{thm:2.10}
	Let $(M,h)$ be a complete Kähler manifold and let $e^{-\psi}$ be a weight on~$M$. Assume that $\partial_p$ is densely defined in the Bergman space $A^2_{(p,0)}(M,h,e^{-\psi})$. If $(\bar{\partial} \psi)^{\sharp}$ is holomorphic, then 
	\begin{equation}\label{e:2.47}
	D^{\ast} \eta = \partial^{\ast}\eta
	\end{equation} 
	for all $\eta \in \dom(D^{\ast}) \cap A^2_{(p + 1,0)}(M,h,e^{-\psi})$.
\end{theorem}
One can state a version of this corollary for general Hermitian non-Kähler manifold. However, when $p >1$, the hypothesis is more technical due to the presence of the term $T^{\sharp}$ in \cref{e254a} above. In Section 5, we exhibit an example of a non-K\"ahler metric such that \cref{e:2.47} holds for $p=1$ and $p=2$.

\section{The complex Laplacian and the basic estimate}
In this section, we study the complex Laplacian associated to the $\partial$-operator restricted to the Bergman spaces. Let $(M,h)$ be a Hermitian manifold and let $e^{-\psi}$ be a weight on $M$. Suppose that $\dom(\partial)$ is dense in $A^2_{(p,0)}(M,h,e^{-\psi})$ and $A^2_{(p-1,0)}(M,h,e^{-\psi})$. Then the Laplacian
\begin{equation}
\boxop_p = \partial \partial^{\ast} + \partial^{\ast} \partial
\end{equation}
is well-defined (for $p = 0$, we define $\boxop_0 = \partial^{\ast} \partial$). This operator was studied earlier in \cite{haslinger} for the case $M= \bC^n$, $h$ is the {Euclidean} metric, and $\psi = |z|^2$. 
Under the density assumption, $\boxop$ acts as an (bounded or unbounded) self-adjoint operator on the Bergman space $ A^2_{(p,0)}(M, h, e^{-\psi})$. We point out that since the compactly supported forms can not have holomorphic coefficients, they are not useful for several problems considered here such as the density of $\dom(\partial)$. In particular, it is a nontrivial question whether $\partial$ is densely defined on the weighted Bergman spaces. Fortunately, in several interesting cases when $M$ is an open subset of $\mathbb{C}^n$, we can use $(p,0)$-forms with polynomial coefficients as a substitute to prove that $\partial$ is densely defined.
\begin{definition}[Basic estimate]
	Let $(M,h,\psi)$ be a Hermitian manifold with smooth weight $e^{-\psi}$ such that the $\partial$-operator is densely defined in $A^2_{(p,0)}(M,h,e^{-\psi})$. We say that the $\partial$-complex satisfy the \emph{basic estimate} on holomorphic $(p,0)$-forms if for each $u \in \dom(\partial_p) \cap \dom(\partial^{\ast}_p)$, we have
	\begin{equation}\label{e:basicestimate}
	\|\partial u\|^2_{h,\psi} + \|\partial^{\ast} u\|^2_{h,\psi} \geqslant c\, \|u\|^2_{h,\psi}
	\end{equation}
	for some constant $c >0$.
\end{definition}
Similarly to the $L^2$-theory for the $\bar{\partial}$-complex, the basic estimate \cref{e:basicestimate} implies various useful properties for the complex Laplacian $\boxop$ (cf. Chapter 8 of \cite{haslinger/book}).

In the following, we describe a simple situation in which the basic estimate for $\partial$-complex holds. For this purpose, we first let $\Theta$ be the Chern-Ricci form of $h$, i.e., $\Theta = -i \partial \bar{\partial} \log \det (h_{j\bar{k}})$ in local coordinate system (see \cite{kodaira--morrow}). For a $(1,0)$-form $u = u_j dz^j$, we define
\begin{equation} 
T u = T^{i}_{jk} u_i\, dz^j \otimes dz^k.
\end{equation} 
The following version of the well-known basic identity generalizes \cite[Theorem~3.2]{haslinger} and has a similar form to \cite[Proposition~5.2]{berger--dallara--son}, cf. \cite{griffiths}. As in the $L^2$-theory of the $\bar{\partial}$, this identity can be used to proved the basic estimate in several situations.
\begin{proposition}[Basic identity] Let $(M,h)$ be a complete Hermitian manifold. Suppose that $\dom(\partial)$ is dense in $A^2_{(1,0)}(M,h,e^{-\psi})$. If $u$ is a $(1,0)$-form in $\dom(\partial^{\ast})$ such that $ \langle u, \partial \psi - \tau \rangle_h \in L^2 (M,h,e^{-\psi})$ and $\partial u - T u \in L^2_{(1,0)}(M,h,e^{-\psi})$, then
	\begin{equation}\label{e:basic}
	\|\partial u - T u\|^2 + \| \partial^\ast u\| ^2
	=
	2 \|\nabla u\|^2 + ( i\partial \bar{\partial} \psi + \Theta, u \wedge \bar{u} )_{h,\psi} -\|(I-P_{h,\psi})( \langle u, \partial \psi - \tau \rangle_h)\|^2 .
	\end{equation}
\end{proposition}
\begin{proof} By \cref{prop:31}, $u\in \dom(D^{\ast})$ and $D^{\ast} u = \langle u , \partial \psi -\tau \rangle_h$. Thus,
	\begin{equation} 
	\partial^{\ast} u = P_{h,\psi} \left(D^{\ast} u\right) = P_{h,\psi}\left(\langle u , \partial \psi - \tau\rangle _h\right).
	\end{equation} 
	Consequently,
	\begin{align}
	\|\partial u - T u\|^2 + \| \partial^\ast u\| ^2
	=
	\|\partial u - T u\|^2 + \| D^\ast u\| ^2 - \|(I - P_{h,\psi}(D^{\ast}u))\|,
	\end{align}
	and therefore, \cref{e:basic} follows, after complex conjugation, from the  well-known identity for the $\bar{\partial}$-complex (see, e.g., \cite{berger--dallara--son} or \cite{griffiths}).
\end{proof}
If $(M,h)$ is not a complete manifold but a relatively compact domain in a complex manifold with smooth boundary, we can still formulate a similar basic identity with boundary term. We shall not use such an identity and hence omit the details.

Next, we define the torsion $(1,1)$ form as follows (cf. \cite{berger--dallara--son}).
\begin{equation} 
T \circ \overline{T}
:=
h^{a\bar{\ell}}h^{b\bar{m}}h_{\bar{q} j} h_{p\bar{k}}\, T^{p}_{a b}\, \overline{T^{q}_{\ell m}}\, dz^j\wedge dz^{\bar{k}}.
\end{equation} 
Observe that $\langle i T \circ \overline{T} , \overline u \wedge u\rangle_h = |Tu|^2$.
\begin{corollary}\label{cor:coercive}
	Let $(M,h)$ be a complete Hermitian manifold and $e^{-\psi}$ a smooth weight on~$M$. Suppose that the
	following conditions hold.
	\begin{enumerate}
		\renewcommand{\labelenumi}{(\roman{enumi})}
		\item $\dom(\boxop)$ is dense in $A^2_{(1,0)}(M,h,e^{-\psi})$,
		\item $\partial \psi - \tau \in L^2_{(1,0)}(M,h,e^{-\psi})$ and $(\bar{\partial} \psi - \bar{\tau})^{\sharp}$ is holomorphic,
		\item on $M$,
		\begin{equation}\label{e:curvature}
		i\partial\bar{\partial} \psi + \Theta - i \left(\frac{\sigma}{\sigma - 1}\right) T\circ \bar{T} \geqslant \sigma b^2 \omega_h,
		\  \text{for some}\ b> 0\ \text{and} \ \sigma >1.
		\end{equation}	
	\end{enumerate}
	If $Tu \in L^2_{(1,0)}(M,h,e^{-\psi})$, then 
	\begin{equation}\label{e:coercive}
	\|\partial u \|^2 + \| \partial^\ast u\| ^2 \geqslant b^2 \|u\|^2.
	\end{equation}
\end{corollary}
\begin{proof}
	If $(\bar{\partial} \psi - \bar{\tau})^{\sharp}$ is holomorphic, then $\langle u, \partial\psi -\tau \rangle$ is holomorphic. Thus,
	\begin{equation}
	(I - P_{h, \psi})(\langle u , \partial \psi - \tau \rangle) =0.
	\end{equation}
	The basic identity reduces to
	\begin{equation}
	\|\partial u - T u\|^2 + \|\partial^{\ast} u\|^2 = 2\|\nabla u\|^2 + \left(i\partial\bar{\partial} \psi + \Theta, u \wedge \bar{u}\right)_{h,\psi}.
	\end{equation}
	From this, we can argue similarly to \cite{berger--dallara--son} to obtain \cref{e:coercive}. We omit the details.
\end{proof}
We conclude this section by pointing out that although the basic identity is very useful to establish the basic estimate for the Laplacian associated to a $\bar{\partial}$-complex, the condition \cref{e:curvature} in \cref{cor:coercive} seems to be rather strong (compared to the situation for $L^2$-complex) for the $\partial$-complex in Bergman spaces since we only require \cref{e:basicestimate} to hold for holomorphic $(p,0)$-forms. In \cref{sec:unitball}, we shall meet two situations in which \cref{e:curvature} either fails or holds with a non-optimal constant.

\section{The $\partial$-Neumann operator on weighted Bergman spaces}
In this section, we assume that for our $\partial$-complex, the basic estimate $\cref{e:basicestimate}$ holds. \cref{cor:coercive} in the last section provides a concrete condition for this to be true, however, there are several situations in which the basic estimate can be proved directly (see \cref{sec:unitball}). In these situations, it is natural to study
the bounded inverse $\widetilde{N} : = \boxop^{-1}$. More precisely, 
\begin{proposition}
	Suppose that $\dom(\partial_p)$ is dense in $A_{(p,0)}(M,h,e^{-\psi})$ for $p=0,1$ and suppose that the basic estimate \cref{e:basicestimate} holds. Then $\partial$ and $\partial^{\ast}$ have closed ranges. If we endow $\dom(\partial) \cap \dom(\partial^{\ast})$ with the graph norm 
	\begin{equation}
	f \mapsto \left(\|\partial f\|^2 + \|\partial^{\ast} f\|^2 \right)^{\frac{1}{2}}
	\end{equation}
	then the subspace $\dom(\partial) \cap \dom(\partial^{\ast})$ becomes a Hilbert space.
\end{proposition}
\begin{proof}
	As usual $\ker \partial= (\image \partial^*)^\perp$. Therefore,
	\begin{equation}
	(\ker \partial )^\perp = \overline{ \image \partial^{\ast}} \subseteq \ker \partial^{\ast}.
	\end{equation}
	If $u\in \ker \partial \cap \ker \partial^{\ast},$ we have by \cref{e:basicestimate} that $u=0$. Hence
	\begin{equation}\label{eq: ima1}
	(\ker \partial)^\perp = \ker \partial^{\ast}.
	\end{equation}
	If $u\in \dom (\partial) \cap (\ker \partial )^\perp,$ then $u\in \ker \partial^*,$ and \cref{e:basicestimate} also implies
	\begin{equation}
	\|u\| \leqslant \frac{1}{b} \, \| \partial u \|.
	\end{equation}
	To conclude the proof, we can use general results of unbounded operators on Hilbert spaces (see for instance \cite[Chapter~4]{haslinger/book}) to show that $\image\partial $ and $\image \partial^{\ast}$ are closed. The last assertion follows again by \cref{e:basicestimate}.
\end{proof}
\begin{theorem}\label{thm:sur} 
	Suppose that $\dom(\partial_p)$ is dense in $A_{(p,0)}(M,h,e^{-\psi})$ for $p=0,1$ and suppose that the basic estimate \cref{e:basicestimate} holds. Then the operator $\boxop \colon \dom(\boxop) \rightarrow A^2_{(1,0)}(M,h,e^{-\psi})$ is bijective and has a bounded inverse
	\begin{equation}
	\widetilde{N} \colon A^2_{(1,0)}(M,h,e^{-\psi}) \to \dom(\boxop).
	\end{equation}
	In addition 
	\begin{equation}\label{cont5}
	\|\widetilde{N} u\| \leqslant \frac{1}{b}\, \|u\|,
	\end{equation}
	for each $u\in A^2_{(1,0)}(M,h,e^{-\psi})$.
\end{theorem}
\begin{theorem}\label{invers}
	Suppose that $\dom(\partial_p)$ is dense in $A^2_{(p,0)}(M,h,e^{-\psi})$ for $p=0,1$ and suppose that the basic estimate \cref{e:basicestimate} holds.
	Let $\eta \in A^2_{(1,0)}(M, h, e^{-\psi})$ with $\partial \eta =0.$ 
	Then $u_0: =\partial^{\ast} \widetilde{N}_1\eta$ is the canonical solution of $\partial u =\eta. $ This means $\partial u_0 =\eta $ and $u_0 \in (\ker \partial )^\perp$. Moreover, 
	\begin{equation}\label{cont6}
	\left\| \partial^{\ast} \widetilde{N}\eta \right\| \leqslant b^{-1/2} \, \| \eta \| . 
	\end{equation}
\end{theorem}
\begin{example}\label{ex:44} Consider the complex plane $\mathbb{C}$ and a radial weight function 	$\psi(|z|^2)$, where $\psi(t)$ is a real-valued function of one real variable. Suppose that $\psi' > 0$ and put $h = \psi'(|z|^2) dz \otimes d\bar{z}$. Then $d\mu = e^{-\psi(|z|^2)} \dvol_h = e^{-\psi(|z|^2)} \psi'(|z|^2) d\lambda$. Clearly, $(\bar{\partial}\psi)^{\sharp} = z\partial_z$ is a holomorphic vector field. The special case $\psi(t) =t$ leads to the case of {Segal--Bargmann space} and was studied thoroughly in \cite{haslinger}. 
	
	For a constant $\alpha \geqslant 2$, we put $\psi(z) = \alpha \log (1+|z|^2)$ and consider the complete metric
	\begin{equation}
	h = (1+|z|^2)^{-1} dz \otimes d\bar{z}.
	\end{equation}
	(This metric is often referred to as Hamilton's cigar soliton in the literature). Put
	\begin{equation}
	e^{-\psi} \dvol_h = (1+|z|^2)^{-(\alpha+1)} d\lambda,
	\end{equation}
	where $d\lambda$ is the standard Lebesgue measure. The Bergman space $A^2(\mathbb{C}, e^{-\psi} \dvol_h)$ is the \textit{finite} dimensional space of polynomials of degree $\leqslant \alpha -1$.
	
	If $u(z) \in A^2(\mathbb{C}, e^{-\psi} \dvol_h)$, i.e., $u(z)$ is a polynomial of degree $k\leqslant \alpha -1$, then
	\begin{equation}
	|\partial u|^2_h
	=
	|u'(z) dz|^2_h
	=
	|u'(z)|^2 (1+|z|^2).
	\end{equation}
	Since $u'(z)$ is of degree $k-1$, we can easily see that then $\partial u \in A^2_{(1,0)}(\mathbb{C}, h, e^{-\psi})$. That is, $\partial$ maps $ A^2(\mathbb{C}, h, e^{-\psi})$ into $ A^2_{(1,0)}(\mathbb{C}, h, e^{-\psi})$. This is an operator between finite dimensional Hilbert spaces and hence bounded. Notice that $|\partial \psi|_h = \alpha |z|/\sqrt{1+|z|^2}$ is also bounded (cf. \cref{prop:31}).
	
	Let $D$ be the $\partial$-operator on the weighted $L^2$ spaces, then
	\begin{equation}
	D^{\ast} (udz)
	=
	- h^{-1}\partial_{\bar{z}} u 
	-u \partial_{\bar{z}} h^{-1}
	+ h^{-1} (\partial_{\bar{z}} \varphi) u
	=
	- (1+|z|^2) \partial_{\bar{z}} u + \alpha z u.
	\end{equation}
	Therefore, the weighted Bergman space adjoint is a ``multiplication'' operator:
	\begin{equation}
	\partial^{\ast}(udz)
	=
	\alpha zu,
	\end{equation}
	where $u$ is a holomorphic polynomial of degree $\leqslant \alpha -2$.
	
	The formula for $\boxop $ is rather simple. Indeed, for a polynomial $f \in A^2(\mathbb{C}, h, e^{-\psi})$, one has
	\begin{equation}
	\boxop _0 f = \alpha z f'(z).
	\end{equation}
	Clearly, $\ker \boxop_0$ is the space of constants. Moreover, $\alpha, 2\alpha,\dots , (\lfloor \alpha\rfloor-1)\alpha$ are the eigenvalues with corresponding eigenvectors $z, z^2,\dots, z^{\lfloor \alpha\rfloor-1}$, respectively.
	
	For a holomorphic $(1,0)$-form $udz$, one has
	\begin{equation}
	\boxop _1 (udz) = \alpha (u + zu') dz.
	\end{equation}
	Thus, $\ker\boxop_1 = \{0\}$. Moreover, the eigenvalues are $\alpha k$ for $k = 1,\dots, \lfloor \alpha\rfloor -1$. The basic identity gives $\boxop_1 \geqslant (\alpha + 1)/2$, as quadratic forms, but in fact we have a stronger inequality $\boxop_1 \geqslant \alpha$. Observe that $\boxop_1$, expected to be of second order, is actually of first order when restricted to 	holomorphic forms.
\end{example}

\section{Weighted Bergman spaces on the unit ball of $\mathbb{C}^n$}\label{sec:unitball}

In the sequel, we shall focus on the unit ball $\mathbb{B}^n \subset \mathbb{C}^n$, although the case of $\mathbb{C}^n$ is interesting as well; see \cite{haslinger} and several examples below. In the first part, we investigate $U(n)$-invariant K\"ahler metrics and radial weights. The most important example in this case is the complex hyperbolic metric with a radial weight making the duality between differentiation and multiplication similar to the case of the Segal-Bargmann space. In the second part of this section, we consider $U(n)$-invariant conformally flat Hermitian metrics and radial weights. We focus on a special case for which \cref{thm:2.10} still holds. This case leads to a proof of \cref{thm:1.1} stated in the introduction. In both parts, we shall use, as in \cite{haslinger}, the following result.
\begin{lemma}[see Lemma~1.2.2 of \cite{davies1995spectral}]\label{lem51}
	Let A be a symmetric operator on a Hilbert space $H$ with domain $\dom(A)$ and suppose that $\{x_k\}_k$ is a complete orthonormal system in $H$. If each $x_k$ lies in $\dom(A)$ and there exists $\lambda_k \in \mathbb{R}$ 
	such that $Ax_k = \lambda_k x_k$ for every $k \in \mathbb{N}$, then $A$ is essentially self-adjoint and the spectrum of $\overline A$ is the closure in $\mathbb{R}$ of the set of all $\lambda_k$.
\end{lemma}
\subsection{$U(n)$-invariant K\"ahler metrics and radial weights}
In the sequel, we consider $U(n)$-invariant K\"ahler metrics and radial weights. To this end, we suppose that $h_{j\bar{k}}$ is a K\"ahlerian metric induced by a radial potential $h(z) = \tilde{h}(|z|^2)$, where $\tilde{h}(r)$ is a real-valued function of a real variable. Precisely, we have
\[
h_{j\overline{k}} = \partial_{j} \partial_{\overline{k}}\, \tilde{h}(|z|^2)
=
\tilde{h}'(|z|^2)\, \delta_{jk} + \tilde{h}''(|z|^2)\,  \overline{z}_j z_k.
\]
Thus, $h_{j\bar{k}}$ is a rank-one perturbation of a multiple of the identity matrix. For $h_{j\overline{k}}$ to be positive definite, we assume that $\tilde{h}'(r) > 0$ and $r\tilde{h}'' + \tilde{h}' > 0$. The Sherman-Morrison formula give the formula for the (transposed) inverse 
\[
	h^{k \overline{j}}
	=
	\frac{1}{\tilde{h}'} \left(\delta_{jk} - \frac{\tilde{h}'' z_k \overline{z}_j}{\tilde{h}' + r \tilde{h}''}\right), \quad r = |z|^2,
\]
so that $h_{\overline{l} k} h^{k\overline{j}} = \delta_{\overline{l}}^{\overline{j}}$, the Kronecker symbol.

If $\psi(z) = \tilde{\psi}(r)$, $r = |z|^2$, is a radial weight, then $\partial_{\overline{j}} \psi = \tilde{\psi}'(r) z_j$. Hence
\[
\left(\overline{\partial} \psi\right)^{\sharp} = \left(\frac{\tilde{\psi}'}{\tilde{h}' + r \tilde{h}''}\right) \sum_{k=1}^{n} z_k \frac{\partial}{\partial z_k}.
\]
The holomorphicity of $\left(\overline{\partial} \psi\right)^{\sharp}$ is equivalent to 
\[
\frac{\tilde{\psi}'}{\tilde{h}' + r \tilde{h}''} = C,
\]
for some constant $C$. This is equivalent to
\begin{equation}\label{e:last}
\tilde{\psi} = C r \tilde{h}' + \tilde{C},
\end{equation}
for some constants $C$ and $\tilde{C}$. From this equation, we can construct many examples of K\"ahlerian metrics which are invariant with respect to $U(n)$-action and radial weight for which \cref{thm:2.10} holds.

The most basic and interesting example that can be constructed as described above is that of the complex hyperbolic metric. Precisely, consider the unit ball $\mathbb{B} \subset \mathbb{C}^n$ endowed with the Bergman--{K\"ahler} metric:
\begin{equation}
h_{j\bar{k}}
=
-\partial_j \partial_{\bar{k}} \log (1-|z|^2)
=
(1-|z|^2)^{-1}\delta_{jk} + (1-|z|^2)^{-2} \bar{z}^j z^k.
\end{equation}
Here, $|z|^2 := \sum_{j=1}^{n} |z^j|^2$. From \cref{e:last} with $C = \tilde{C} = \alpha$, we let the weight function be
\begin{equation}
\psi(z) = \frac{\alpha }{1-|z|^2},
\quad 
\alpha >0,
\end{equation}
so that	
\begin{equation}
d\mu = e^{-\psi} \dvol_h 
=
(1-|z|^2)^{-n-1} \exp \left(-\frac{\alpha}{1-|z|^2}\right) d\lambda.
\end{equation}
It turns out that this Bergman space with the so-called ``exponential weight'' has duality properties similar to the Segal-Bargmann space, so it can be seen as a version of the Segal-Bargmann space on a bounded domain. We will show that the adjoint of the densely defined unbounded operator $\partial$ is the operator multiplication by $\alpha z.$ But in this case we have to take care of the Hermitian metric on $\mathbb B$ and of the fact that $\partial $ maps $A^2_{(p,0)}( \mathbb B, h , d\mu )$ into $A^2_{(p+1,0)}( \mathbb B, h , d\mu )$ and $\partial^*$ maps $A^2_{(p+1,0)}( \mathbb B, h , d\mu )$ into $A^2_{(p,0)}( \mathbb B, h , d\mu ).$

Since the weight is radial, the polynomials are dense in Bergman space $A^2_{(p,0)}(\mathbb{B}, h, \psi)$ (see \cite{mergelyan1953completeness} or the proof of Proposition 2.6 in \cite{zhu2005spaces} for the case $p=0$ with ``standard'' weight; the case of general radial weights and $p\geqslant 1$ follows easily). Moreover, the monomials $z^J$'s (each $J$ is a multiindex) are orthogonal in $A^2(\mathbb{B}, d\mu)$. For each $k$, put
\begin{equation}\label{coeff}
a_k = \int_0^1 (1-t)^{-n-1} \exp \left(-\frac{\alpha }{1-t}\right) t^{n+k-1} dt =
\frac{1}{\alpha^n}\int_0^\infty \frac{s^{n+k-1}}{(\alpha +s)^k} \,e^{-s - \alpha} \,ds.
\end{equation}
Then by the density of the polynomials, an orthonormal basis for $A^2(\mathbb{B}, d\mu)$ can be taken as
\begin{equation}
e_J : = \frac{(|J| + n - 1)!}{\sqrt{a_{|J|}} \pi^{n/2} J!} \, z^J,
\end{equation}	
where $J$ is a multi-index.	

Observe that
\begin{equation}
h^{j\bar{k}}
=
(1-|z|^2)(\delta_{jk} - z^j \bar{z}^k).
\end{equation}
Therefore, if $u = u_j dz^j$, then
\begin{equation}
|u|^2_h 
=
u_j u_{\bar{k}} h^{j\bar{k}}
=
(1-|z|^2)\left(\sum_{j=1}^n |u_j|^2 - \sum_{j,k} z^j\bar{z}^k u_j u_{\bar{k}}\right).
\end{equation}
Hence, the holomorphic $(1,0)$-forms with polynomial coefficients 
are in $L^2_{(1,0)}(\mathbb{B}, h, e^{-\psi})$.

We also compute,
\begin{equation}
\psi_{\bar{k}}
=
\frac{\alpha z^k}{(1-|z|^2)^2},
\end{equation}
and find that (sum over $k$)
\begin{equation}\label{e:59}
h^{j\bar{k}}\psi_{\bar{k}} = \alpha z^j
\end{equation}
are holomorphic. Consequently, for $u = u_j dz^j$
\begin{equation}
\langle u, \partial \psi\rangle_h
=
\alpha z^j u_j.
\end{equation}
Thus, if $u_j$'s are holomorphic polynomials, then $\langle u, \partial \psi\rangle_h$ is a holomorphic polynomial and hence $u\in \dom (\partial^{\ast})$. Moreover, {by \cref{prop:31}}
\begin{equation}\label{e:511}
\partial^{\ast} (u_j dz^j)
=
P_{h,\psi}(\langle u, \partial\psi\rangle_h)
=
\alpha z^j u_j.
\end{equation}
Since the restrictions of polynomials are dense in $L^2(\mathbb{B}, h,e^{-\psi})$, formula \cref{e:511} for $\partial^{\ast}$ holds for every $u \in \dom(\partial^{\ast})$.

Using Taylor series expansion (in sake of simplicity we take $n=1$) we can directly verify that for $f\in \dom(\partial)$ and $g\, dz\in \dom(\partial^{\ast})$. We have 
\begin{equation}\label{duality}
(\partial f, g\, dz )_{h,\psi} = ( f'dz, g\,dz)_{h,\psi} =( f, \alpha zg)_{h,\psi}.
\end{equation}
Each $f\in A^2(\mathbb{B}, h, d\mu)$ can be represented in the form $f = \sum_{k=0}^\infty f_k e_k,$ where 
$e_k= \frac{z^k}{c_k}$ and
$(f_k)_k \in l^2$ and each $F\in A^2_{(1,0)}(\mathbb{B}, h, d\mu)$ can be represented in the form $F= \left(\sum_{k=0}^\infty F_k E_k\right) dz,$ where $E_k = \frac{z^k}{d_k}$ and $(F_k)_k \in l^2.$
We write $f = \sum_{k=0}^\infty f_k e_k,$ where $(f_k)_k \in l^2,$ and 
$g = \sum_{k=0}^\infty g_k e_k,$ where $(g_k)_k \in l^2$ and have
\begin{equation}
c_k^2 = 2\pi \int_0^1 r^{2k+1} (1-r^2)^{-2} \exp \left ( \frac{-\alpha}{1-r^2} \right )\,dr = \frac{\pi}{\alpha} \int_0^\infty \left ( \frac{s}{s+\alpha} \right )^k \, e^{-s - \alpha}\,ds,
\end{equation}
and
\begin{equation}
d_k^2= 2\pi \int_0^1 r^{2k+1} \exp \left ( \frac{-\alpha}{1-r^2} \right )\,dr = \pi \alpha 
\int_0^\infty \frac{s^k}{(s+\alpha)^{k+2}} \, e^{- s -\alpha}\,ds.
\end{equation}
Now we obtain 
\begin{align}
\left( f'dz, gdz \right)_{h,\psi} 
& = \int_{\mathbb B} f'(z) \overline{g(z)} \exp \left ( \frac{-\alpha}{1-|z|^2} \right )\, d\lambda \notag \\
& = \sum_{k=0}^\infty f_{k+1}(k+1) \frac{1}{c_k c_{k+1}} \overline g_k d_k^2.
\end{align}
On the other hand,
\begin{equation}
( f, \alpha zg)_{h,\psi} = \alpha \sum_{k=0}^\infty f_{k+1} \frac{c_{k+1}}{c_k} \overline g_k.
\end{equation}
In order to prove \eqref{duality} we have to show that 
\begin{equation}\label{norms}
\alpha c_{k+1}^2 = (k+1) d_k^2.
\end{equation}
Observe that 
\begin{equation}
\frac{d}{ds}\left [ \left ( \frac{s}{s+\alpha} \right )^{k+1} \right ] = (k+1)\left ( \frac{s}{s+\alpha} \right )^{k} \, \frac{\alpha}{(s+\alpha)^2}.
\end{equation}
Using integration by parts we get 
\begin{align}
\alpha\, c_{k+1}^2 
& = \pi \int_0^\infty \left ( \frac{s}{s+\alpha} \right )^{k+1} \, e^{-s - \alpha}\,ds \notag \\
& = \pi \alpha (k+1) 
\int_0^\infty \frac{s^k}{(s+\alpha)^{k+2}} \, e^{-s - \alpha }\,ds = (k+1)\, d_k^2.
\end{align}
Thus, \cref{norms} and \cref{duality} follow.

We indicate that $\partial: A^2(\mathbb{B},h, d\mu) \longrightarrow A^2_{(1,0)}(\mathbb{B}, h, d\mu)$ is an unbounded operator. It suffices to consider the one-dimensional case (i.e., $n=1$). 

Take $g=\sum_{k=1}^\infty \frac{1}{k}e_k,$ then $g\in A^2(\mathbb{B}, h, d\mu),$
but $\partial (g)= g'\, dz$, where 
\begin{equation}
g' =\sum_{k=1}^\infty \frac{1}{k} \frac{k d_{k-1}}{c_k} \frac{z^{k-1}}{d_{k-1}}= 
\sum_{k=0}^\infty H_k E_k,
\end{equation} 
and $H_k=(d_{k}/c_{k+1})= \sqrt{\frac{\alpha}{k+1}},$ by \eqref{norms}, which implies that $g'\, dz\notin A^2_{(1,0)}(\mathbb{B}, h, d\mu).$

Similarly, we show that $\partial^* : A^2_{(1,0)}(\mathbb{B}, h, d\mu) \longrightarrow A^2(\mathbb{B},h, d\mu)$ is unbounded. Let $g\, dz = \sum_{k=0}^\infty \frac{1}{k+1} \frac{z^k}{d_k}.$ Then $g\, dz \in A^2_{(1,0)}(\mathbb{B}, h, d\mu).$ But 
\begin{equation}
\partial^*( g\, dz) = \alpha z g = \alpha \sum_{k=0}^\infty \frac{1}{k+1} \frac{c_{k+1}}{d_k} \frac{z^{k+1}}{c_{k+1}} = \alpha \sum_{k=1}^\infty \frac{1}{k} \frac{c_{k}}{d_{k-1}} e_k = \sum_{k=1}^\infty h_k e_k,
\end{equation}
where $h_k=\sqrt{\frac{\alpha}{k}},$ by \eqref{norms}, which implies that $ \partial^*( g\, dz)\notin A^2(\mathbb{B}, h, d\mu)$, as desired. 
\begin{remark} Since $\partial$ is unbounded, $|\partial \psi|_h$ must be unbounded by \cref{prop:31}. One can also see this by direct calculation: $|\partial \psi|^2_{h} = \alpha |z|^2(1-|z|^2)^{-2}$ is unbounded in $\mathbb{B}$.
\end{remark}

Next, we compute $\boxop_1$ for general $n.$ For this, we let $v \in A^2_{(2,0)}(\mathbb{B}, h, d\mu)$ and write 
\begin{equation}
v = \frac{1}{2}\sum_{j,k} v_{jk} dz^j \wedge dz^k = \sum_{j<k} v_{jk} dz^j \wedge dz^k,
\end{equation}
where $v_{jk} = - v_{kj}$ are holomorphic. By \cref{e:243,e:59}, 
\begin{equation}
\partial^{\ast} v
=
- \sum_{j,k,l} v_{kj} \psi_{\bar{l}} h^{j\bar{l}} dz^k
=
\alpha \sum_{j,k} v_{jk} z^{j} dz^k.
\end{equation}
For $u = u_{j} dz^j$, we have
\begin{equation}
\partial u
=
\frac{1}{2}\sum_{j,k} \left(\frac{\partial u_k}{\partial z^j} - \frac{\partial u_{j}}{\partial z^k}\right) dz^j \wedge dz^k.
\end{equation}
Thus,
\begin{equation}
\partial^{\ast} \partial u
=
\alpha \sum_{j,k} \left(\frac{\partial u_k}{\partial z^j} - \frac{\partial u_{j}}{\partial z^k}\right) z^j dz^k. 
\end{equation}
On the other hand, since
\begin{equation}
\partial^{\ast} u = \alpha \sum_{j} u_j z^j,
\end{equation}
we have
\begin{equation}
\partial \partial^{\ast} u
=
\alpha \sum_{k} \left( u_k + \sum_{j} z^j \frac{\partial u_j}{\partial z^k}\right)\, dz^k.
\end{equation}
Consequently,
\begin{equation}
\boxop_1 u
=
\alpha \left[ u + \sum_{j,k} z^j \frac{\partial u_k}{\partial z_j}dz^k \right].
\end{equation}
This is similar to the formula for $\boxop_1$ on the Segal-Bargmann space given in \cite[Eq. (2.5)]{haslinger}. Thus, we have
\begin{theorem}\label{prop:51} Let $\alpha > 0.$ Then $\boxop_1$ has a bounded inverse $\widetilde{N}_1,$ which is a compact operator on $A^2_{(1,0)}(\mathbb{B}, h , e^{-\psi})$ with spectrum $\{ \alpha k : k\in \mathbb N\},$ where each eigenvalue $\alpha k$ has multiplicity $n \binom{n+k-2}{n-1}.$ 
	In addition we have 
	\begin{equation}
	\left\| \widetilde{N}_1 u \right\| \leqslant \frac{1}{\alpha} \, \|u\|,
	\end{equation} 
	for each $u \in A^2_{(1,0)}(\mathbb{B}, h , e^{-\psi}).$
	
	Consequently, if $\eta = \eta_j dz^j \in A^2_{(1,0)}(\mathbb{B}, h, e^{-\psi})$ with $\partial \eta =0$, then $f: =\partial^{\ast} \widetilde{N}_1 \eta$ is the canonical solution of $\partial f = \eta, $ this means $\partial f = \eta $ and $f \in (\ker \partial )^\perp$. Moreover, 
	\begin{align}\label{cont61}
	\int_{\mathbb{B}} \left|f\right|^2 & (1-|z|^2)^{-n-1} \exp\left(-\frac{\alpha}{1-|z|^2}\right) d\lambda \notag \\
	&\leqslant
	\frac{1}{\alpha} \int_{\mathbb{B}} \left[\sum_{j=1}^n |\eta_j|^2 - \sum_{j,k=1}^{n} \eta_j \overline{\eta_k}z^j \bar{z}^k\right](1-|z|^2)^{-n}\exp\left(-\frac{\alpha}{1-|z|^2}\right)d\lambda.
	\end{align}
	The constant $\alpha^{-1}$ on the right-hand side is sharp. Equality occurs if and only if the $\eta_j$'s are constants.
\end{theorem}
\begin{proof}
	The proof is similar to that of \cite[Theorem 4.8]{haslinger} and uses \cref{invers}. Indeed, the coercivity of $\boxop_1$ follows directly from the fact that its spectrum consists of the the point eigenvalues $\alpha k$, $k = 1,2,\dots,$ each with finite multiplicity. Thus, $\boxop_1$ has a bounded inverse~$\widetilde{N}_1$. 
	
	Since $\eta \in \ker\partial \subset A^2_{(1,0)}(\mathbb{B}, h, \psi)$, we can define $f = \partial^{\ast} \widetilde{N}_1 \eta$. Standard arguments imply that $ f$ is orthogonal to $\ker\partial = \{\mathrm{constants}\} $ and $\partial f = \eta$. Moreover, 
	\begin{equation*} 
	\|f\|^2 
	=
	\left(\partial^{\ast} \widetilde{N}_1 \eta, f\right)_{h,\psi}
	=
	\left(\widetilde{N}_1 \eta, \partial f\right)_{h,\psi}
	=
	\left(\widetilde{N}_1 \eta, \eta\right)_{h,\psi}
	\leqslant \frac{1}{\alpha} \|\eta\|^2,
	\end{equation*} 
	and hence \cref{cont61} follows. Equality occurs if and only if $\eta$ belongs to the eigenspace corresponding to $\lambda_1 = \alpha$ which is spanned by $dz^k$, $k=1,2,\dots, n$, or equivalently, if the $\eta_k$'s are constants. The proof is complete.
\end{proof}

\begin{remark}
	If $n=1$ and $u = u \, dz \in A^2_{(1,0)}(\mathbb{B}, h , e^{-\psi})$, then
	\begin{equation}
	\boxop_1 u = \partial \partial^* u = \alpha \left(u +z \frac{\partial u}{\partial z} \right)\, dz.
	\end{equation}
	By \cref{prop:51}, we have for $\alpha > 0$ that
	\begin{equation}
	( \boxop_1 u, u)_{h,\psi} \geqslant \alpha \|u\|^2.
	\end{equation}
	For any $v\in A^2_{(1,0)}(\mathbb{B}, h , e^{-\psi})$, there exists $w\in A^2_{(0,0)}(\mathbb{B}, h , e^{-\psi})$ such that
	\begin{equation}
	\partial w = v,
	\end{equation}
	and
	\begin{equation}
	\|w\|^2 \leqslant \frac{1}{\alpha} \, \|v\|^2.
	\end{equation}
	The operator $\boxop_1$ is has an bounded inverse. We point out that \cref{invers} only gives a weaker estimate under a stronger assumption $\alpha >2$. This is also the case in higher dimension. To see this, we compute the Ricci form $\Theta = -i \partial \bar{\partial} \log \det (h_{j\bar{k}}) = -(n+1) \omega_h$, here $\omega_h = i h_{j\bar{k}} dz^j \wedge d\bar{z}^k$ is the Kähler form. Therefore, for $\epsilon > 0 $ and $\gamma: = n+1+\epsilon$, 
	\begin{align}
	i\partial\bar{\partial} \psi + \Theta - \epsilon\, \omega_h
	=
	i\left[ \frac{\alpha - \gamma (1 - |z|^2)}{(1-|z|^2)^2}
	\delta_{jk} + \frac{2\, \alpha - \gamma (1-|z|^2)}{(1-|z|^2)^3} \bar{z}^j z^k
	\right] dz^j \wedge d\bar{z}^k.
	\end{align}
	Thus, $i\partial\bar{\partial} \psi + \Theta \geqslant \epsilon\, \omega_h$ on $\mathbb{B}$ if and only if $\alpha \geqslant \gamma$, i.e., $\epsilon \leqslant \alpha - n -1$. However, we can only deduce from \cref{cor:coercive} the basic estimate \cref{e:basicestimate} with constant $c = \alpha - n -1$ which is much smaller than the lowest eigenvalue $\lambda_1 = \alpha$ of~$\boxop_1$.
\end{remark}
\subsection{Conformally K\"ahler metrics}\label{subsec:nonkahler}
In the following, we give several examples of spaces of holomorphic functions and Hermitian metrics (non-K\"ahlerian for $n\geqslant 2$) such that the $(1,0)$-vector field $(\overline \partial \psi - \overline \tau)^\sharp $ is holomorphic and, moreover, for any $(2,0)$-form $\eta \in \dom(D^{\ast})$, $D^{\ast} \eta$ is holomorphic where $\eta$ is holomorphic (The same holds for $(1,0)$-forms by the holomorphicity of $(\overline \partial \psi - \overline \tau)^\sharp $.) These properties allow us to obtain an explicit formula for the complex Laplacian~$\boxop_1$ and to study its spectrum.

To begin with, we consider conformally flat Hermitian metrics of the form
\[
	h_{j\bar{k}} = e^{\varphi} \delta_{jk} 
\]	
where $\varphi(z) = \tilde{\varphi}(|z|^2)$ and $\tilde{\varphi}(r)$ is a real-valued function of one real variable.	Clearly, the torsion tensor and torsion form is
\[
	T^i_{jk} = \tilde{\varphi}'(|z|^2)\left(\bar{z}_j \delta_k^i - \bar{z}_k \delta_j^i\right),
	\quad
	\tau_j = (n-1) \tilde{\varphi}'(|z|^2) \bar{z}_j.
\]
Assume that $\psi(z) = \tilde{\psi}(|z|^2)$ for some real-valued function of one real variable $\tilde{\psi}$, then
\[
	\left(\overline{\partial} \psi - \overline{\tau}\right)^{\sharp}
	= 
	\exp\left(-\tilde{\varphi}(|z|^2)\right) \left(\tilde{\psi}' (|z|^2) - (n-1) \tilde{\varphi}' (|z|^2)\right) z^j \frac{\partial}{\partial z_j}.
\]	
Thus, $\left(\overline{\partial} \psi - \overline{\tau}\right)^{\sharp}$ is holomorphic if and only if
\[
	\exp\left(-\tilde{\varphi}(|z|^2)\right) \left(\tilde{\psi}' (|z|^2) - (n-1) \tilde{\varphi}' (|z|^2)\right)  = C
\]
for some constant $C$, or equivalently,
\begin{equation}\label{e:last2}
	\tilde{\psi} = C \int e^{\tilde{\varphi}} dr + (n-1) \tilde{\varphi}.
\end{equation}

Using \cref{e:last2}, we can easily exhibit some examples on $\mathbb C^2:$
\begin{example} We  choose  $h_{j\overline k}=\delta_{jk} (1+|z|^2)^{m}, $ with $m \geqslant 1,$
	and 
	\[
		\psi (z) =  m \log (1+|z|^2) - \frac{\alpha (1+|z|^2)^{m+1}}{m+1}.
	\] 
	Then we get
	\begin{equation*} 
		(\overline \partial \psi -\overline \tau )^\sharp = - \alpha  \left(z_1 \frac{\partial}{\partial z_1} + z_2 \frac{\partial}{\partial z_2}\right),
	\end{equation*} 
	if $\alpha < 0,$ we get a non-trivial Bergman space $A^2 (\mathbb C^2, h, e^{-\psi}).$ 
	Next we take
	$h_{j\overline k}=\delta_{jk} \exp (|z|^2), $ 
	and
	\[
		\psi (z) =  |z|^2 -\alpha \exp (|z|^2).
	\]
	If $\alpha < 0,$ we get a non-trivial Bergman space $A^2 (\mathbb C^2, h, e^{-\psi})$ and again
	\begin{equation*} 
	(\overline \partial \psi -\overline \tau )^\sharp = - \alpha  \left(z_1 \frac{\partial}{\partial z_1} + z_2 \frac{\partial}{\partial z_2}\right).
	\end{equation*}
	These examples can be easily generalized to the case of arbitrary dimension. We omit the details.
\end{example}

We now focus on the special case leading to the solution of the $\partial$-equation on the standard weighted Bergman spaces on the ball. Precisely, on the unit ball $\mathbb{B}^n \subset \mathbb{C}^n$, we choose $h_{j\overline k}=\delta_{jk} (1-|z|^2)^{-1}$ for $z \in \mathbb{B}^n$ and $\psi (z)= \alpha \log (1-|z|^2)$, $\alpha \in \mathbb{R}$.
An easy computation shows that
\begin{equation} \label{e:tor1}
T^i_{jk}
=
(1 - |z|^2)^{-1} \left(\bar{z}_j \delta_{ik} - \bar{z}_k \delta_{ij}\right),
\quad
\tau_j = T^{i}_{ji} = (n-1)(1- |z|^2)^{-1} \bar{z}_j,
\end{equation}
and
\begin{equation}
\psi_{\bar{j}}
= -\alpha (1 - |z|^2)^{-1} z_j.
\end{equation} 
Therefore, 
\begin{equation}\label{e:539}
(\overline \partial \psi -\overline \tau )^\sharp = (1-n-\alpha ) \sum_{j=1}^{n} z^j\frac{\partial}{\partial z^j}
\end{equation} 
is also holomorphic. In this case we have, for $p \geqslant 0$,
\begin{equation*} 
L^2_{(p,0)}(\mathbb{B}^n, h, e^{-\psi})=\left\{ \sum_{|J| = p}{}^{'} u_J dz^J \colon 
\sum_{|J| = p}{}^{'}\int_{\mathbb{B}^n}|u_J(z)|^2 (1-|z|^2)^{p -n - \alpha}\,d\lambda (z) < \infty \right\}.
\end{equation*} 
For $p = 0$, the Bergman space $A^2_{(0,0)}(\mathbb{B}, h, e^{-\psi})$ coincides with the usual Bergman space $A^2_{\gamma-1}(\mathbb{B})$, with $\gamma = 1 -n -\alpha$ (see, e.g., \cite{zhu2005spaces}). As usual, we assume $\gamma > 0$, or equivalently, $\alpha < 1-n$. Then for each $p$ the weighted Bergman space $A^2_{(p,0)}(\mathbb{B}^n, h, e^{-\psi})$
is closed in the weighted Lebesgue space $L^2_{(p,0)}(\mathbb{B}^n, h, e^{-\psi})$ by Corollary 2.5 in \cite{zhu2005spaces}. Moreover, the monomials $$\left\{\dfrac{z^J}{c_J} \colon |J| \geqslant 0\right\}$$
form an orthonormal basis in $A^2_{(0,0)}(\mathbb{B}^n, h, e^{-\psi}) \cong A^2(\mathbb{B}, (1-|z|^2)^{-n-\alpha} d\lambda)$ and the $(1,0)$-forms with monomial coefficients $$\left\{\dfrac{z^J}{d_J} dz^k \colon |J| \geqslant 0,\, k = 1,2,\dots, n\right\}$$
form an orthonormal basis in $A^2_{(1,0)}(\mathbb{B}^n, h, e^{-\psi}),$ here $J=(j_1,\dots,j_n)$ is a multi-index,
\begin{equation}\label{cc}
c_J^2= \int_{\mathbb{B}} \left|z^J\right| ^2(1 - |z|^2)^{-n-\alpha} d\lambda = \frac{\omega_n \, n!\, J! \, \Gamma (1-n-\alpha)}{\Gamma \left(|J|+1-\alpha\right)},
\end{equation}
and similarly
\begin{equation}\label{dd}
d_J^2= \int_{\mathbb{B}} \left|z^J dz_k\right| ^2_h(1 - |z|^2)^{-n-\alpha} d\lambda
= \frac{\omega_n\, n!\, J! \, \Gamma (2-n-\alpha)}{\Gamma \left(|J|+2-\alpha\right)}.
\end{equation}
Here $\omega_n$ is the volume of the unit ball. Note that the metric $h$ is not complete and hence Andreotti-Vesentini density lemma does not apply. However, in view of the calculations for the hyperbolic metric on $\mathbb B,$ we get for a $(1,0)$-form
$u = \sum_{j=1}^n u_j\, dz_j \in \dom(\partial^*)$ that 
\begin{equation}\label{cdcd}
\partial^*(u) = (1-n-\alpha) \sum_{j=1}^n z^j \, u_j
\end{equation}
if we can show that 
\begin{equation}\label{duall}
(1-n-\alpha) c^2_{J+_k 1} = (j_k +1) d_J^2,
\end{equation}
here $J+_k 1$ denotes the multi-index $(j_1, \dots, j_{k-1}, j_k +1, j_{k+1}, \dots , j_n).$
Equation \eqref{duall} corresponds to \eqref{norms} and follows easily from \eqref{cc} and \eqref{dd}.

To compute $\partial^{\ast}$ for $(2,0)$-forms, we write $v = \frac12 v_{jk} dz^j \wedge dz^k$, $v_{jk} = - v_{kj}$ and use \cref{e:tor1} to obtain
\begin{equation}
\frac{1}{2} T^{\bar{i}}_{\bar{j}\bar{k}} h^{r\bar{j}} h^{s\bar{k}} h_{p\bar{i}} v_{rs} dz^p
=
z^r v_{rs} dz^s.
\end{equation}
This turns out to be holomorphic for holomorphic $(2,0)$-forms $v$. Plugging this and \cref{e:539} into \cref{e:243} (which is valid since the boundary terms in the integration-by-parts argument vanish due to the factor $1-|z|^2$), we find that
\begin{equation}
\partial^{\ast} v
=
(2 - n -\alpha) z^r v_{rs} dz^s.
\end{equation}
Here the orthogonal projection $P_{h,\psi}$ has no effect since the coefficients $v_{rs}$'s are holomorphic. We get, for $\gamma: = 1 -n -\alpha >0$,
\begin{equation}\label{boxx}
\boxop_1 u
=
\gamma\, u + \left[ (1+ \gamma)\sum_{j,k =1}^{n} z^j \frac{\partial u_k}{\partial z_j} - \sum_{j,k=1}^{n} z^j \frac{\partial u_j}{\partial z^k}\right]  dz^k,
\end{equation}
Unlike the case of complex hyperbolic metric, $\boxop_1$ is not diagonal in the basis $\{z^J dz^l \colon |J| \geq 0, l = 1,2,\dots, n\}$.

The subspaces
\[
A^2_{(1,0)}(m) := \mathrm{span}\, \left\{c_J z^J dz^l \colon, |J| = m, l = 1,2,\dots, n \right\}
\]
are invariant under the action of $\boxop_1$. Using \cref{lem51}, we can study the spectrum of $\boxop_1$ by study the spectra of its restrictions onto finite dimensional subspaces $A^2_{(1,0)}(m)$. When $m = 0$, $A^2_{(1,0)}(0)$ is spanned by $dz^1, dz^2, \dots , dz^n$ and $\boxop_1 (dz^k) = \gamma\, dz^k$ and hence $\gamma$ is an eigenvalue for $\boxop_1$. When $m = 1$, $A^2_{(1,0)}(1)$ has dimension $n^2$ and is spanned by $z^j dz^k$, $j,k = 1, \dots n$; For example, if $n=2$ then the matrix representation of $\boxop_1$ in the basis $e_1:= z_1 dz_1, e_2:=z_1 dz_2, e_3 := z_2 dz_1$, and $e_4:=z_2 dz_2$ is the following constant row/column-sum matrix
\[
	\begin{pmatrix}
		2 \gamma  & 0 & 0 & 0 \\
		0 & 2 \gamma +1 & -1 & 0 \\
		0 & -1 & 2 \gamma +1 & 0 \\
		0 & 0 & 0 & 2 \gamma  \\
	\end{pmatrix}
\]
whose eigenvalues are $2(\gamma+1)$ and $2\gamma$, the latter has  multiplicity~3, and the matrix is diagonalizable. In the general case, by straightforward calculations, we obtain that the smallest eigenvalue of $\boxop_1$ on $A^2_{(1,0)}(m)$ is $(m+1)\gamma$ while the largest one is smaller than $\gamma + m(2+\gamma)$, as simple consequences of a theorem of Ger\v{s}gorin \cite{gervsgorin1931abgrenzung}. Moreover, the corresponding matrix is diagonalizable by the self-adjointness of $\boxop_1$. Thus, by \cref{lem51}, the spectrum of $\boxop_1$ consists of point eigenvalues which are those of the finite dimensional restrictions and each has finite multiplicity. Consequently, we obtain the following result which implies \cref{thm:1.1} stated in Section~1.

\begin{theorem}\label{thm:5.3}
	If $\gamma:= 1-n -\alpha >0 $, then $\boxop_1$ is coercive and has bounded inverse $\widetilde{N}_1$, which is a compact operator on $A^2_{(1,0)}(\mathbb{B}^n,h, e^{-\psi})$ with discrete spectrum.
	
	Consequently, for every $\eta_1, \eta_2, \dots, \eta_n \in A^2_{\gamma}(\mathbb{B})$ such that $\partial \eta_j/\partial z^k = \partial \eta_k/\partial z^j$ for every pair $j,k = 1,2, \dots, n$, there exists $f \in A^2_{\gamma}(\mathbb{B})$ such that $\partial f/\partial z^k = \eta_k$ for every $k= 1,2, \dots, n$, and 
	\begin{equation}\label{e:5.48}
	\int_{\mathbb B} |f|^2 (1- |z|^2)^{\gamma-1} d\lambda \leqslant \frac{1}{\gamma} \, \int_{\mathbb{B}} \sum_{k=1}^n |\eta_k|^2 (1-|z|^2)^{\gamma} d\lambda.
	\end{equation}
	The constant $\gamma^{-1}$ on the right-hand side of \cref{e:5.48} is sharp. Equality occurs if and only if the $\eta_k$'s are constants.
\end{theorem}
\begin{proof} The coercivity of $\boxop_1$ and the existence and compactness of $\widetilde{N}_1$ follow directly from the fact that its spectrum consists of the the point eigenvalues with finite multiplicity. 
	
	Define $\eta = \sum_{k=1}^{n} \eta_k dz^k$, with $\eta_k$'s are holomorphic, $\partial \eta = 0$, and
	\begin{equation} 
	\| \eta \|^2 = \int_{\mathbb{B}} |\eta|_h ^2\, d\mu
	=
	\int_{\mathbb{B}} \sum_{k=1}^{n} |\eta_k|^2 (1 - |z|^2)^{\gamma} d\lambda < \infty.
	\end{equation} 
	Then $\eta \in \ker\partial \subset A^2_{(1,0)}(\mathbb{B}, h, \psi)$. Define $f = \partial^{\ast} \widetilde{N}_1 \eta$. Standard arguments imply that $ f$ is orthogonal to $\ker \partial = \{\mathrm{constants}\} $ and $\partial f = \eta$. Moreover, 
	\begin{equation} 
	\|f\|^2 
	=
	\left(\partial^{\ast} \widetilde{N}_1 \eta, f\right)_{h,\psi}
	=
	\left(\widetilde{N}_1 \eta, \partial f\right)_{h,\psi}
	=
	\left(\widetilde{N}_1 \eta, \eta\right)_{h,\psi}
	\leqslant \frac{1}{\gamma} \|\eta\|^2.
	\end{equation}
	The last inequality follows from the fact that the lowest eigenvalue of $\boxop_1$ is $\lambda_1 = \gamma$. The proof is complete.
\end{proof}
\begin{remark}
We point out again that the usual basic identity as in \cref{cor:coercive} is not useful for the metric $h_{j\bar{k}} = (1-|z|^2)^{-1}\delta_{j\bar{k}}$ as above for $n\geqslant 2$ and hence the usual strategy of the $L^2$-estimate for $\overline{\partial}$-equation fails to yield this result. To see this, we compute,
\begin{equation}
i\partial \bar{\partial} \psi + \Theta
=
i(n+\alpha) \partial\bar{\partial} \log (1-|z|^2)
\end{equation}
and
\begin{equation}
i\, T \circ \overline{T}
=
\frac{2(|z|^2 - \bar{z}^j z^k)}{(1 - |z|^2)^2} dz^j \wedge d\bar{z}^k.
\end{equation}
Consequently, for any $\mu > 1$,
\begin{multline*}
i\partial \bar{\partial} \psi + \Theta - \mu i\, T \circ \overline{T} - \epsilon\, \omega_h\\
= i\left[\frac{(n + \alpha+\epsilon - \mu)|z|^2 - n -\alpha -\epsilon}{(1-|z|^2)^2} \delta_{jk} + \frac{(2\mu -n - \alpha ) \bar{z}^j z^k}{(1- |z|^2)^2} \right] dz^j \wedge d\bar{z}^k.
\end{multline*}
For this to be nonnegative at the origin, $ n+\alpha + \epsilon < 0$. But near the boundary, the hermitian matrix in the bracket on the right-hand side is a rank-one perturbation of the negative constant multiple of the identity matrix and hence can not be nonnegative.
\end{remark}

\vskip 0.5 cm
{\bf Acknowledgement.} The authors thank the referee for several useful suggestions.

\end{document}